\documentclass[11pt,twoside]{article}
\usepackage{amsmath,amssymb,epsfig,subfigure,graphicx,epstopdf,color}
\usepackage{multirow}
\usepackage{algorithm}
\usepackage{algorithmic}
\usepackage[titletoc,title]{appendix}
\newtheorem{proposition}{Proposition}[section]
\newtheorem{theorem}[proposition]{Theorem}
\newtheorem{lemma}[proposition]{Lemma}

\newtheorem{definition}[proposition]{Definition}
\newtheorem{remark}[proposition]{Remark}
\newtheorem{example}[proposition]{Example}

\newenvironment{proof}{{\noindent \em Proof.}}{\hfill $\fbox{}$ \vspace*{5mm}}
\numberwithin{equation}{section}

\newcommand{\ba}{{\bf a}}
\newcommand{\bb}{{\bf b}}
\newcommand{\bc}{{\bf c}}
\newcommand{\bfd}{{\bf d}}

\newcommand{\bfo}{{\bf 1}}

\newcommand{\bn}{{\bf n}}

\newcommand{\bq}{{\bf q}}
\newcommand{\br}{{\bf r}}

\newcommand{\bx}{{\bf x}}
\newcommand{\by}{{\bf y}}
\newcommand{\bz}{{\bf z}}

\newcommand{\bv}{{\bf v}}
\newcommand{\bw}{{\bf w}}

\newcommand{\cb}{{\cal B}}
\newcommand{\cc}{{\cal C}}
\newcommand{\cd}{{\cal D}}

\newcommand{\ck}{{\cal K}}
\newcommand{\cl}{{\cal L}}

\newcommand{\ob}{{\cal OB}}

\newcommand{\cs}{{\cal S}}

\newcommand{\la}{\lambda}

\newcommand{\E}{{\mathbb{E}}}

\newcommand{\R}{{\mathbb{R}}}
\newcommand{\Rm}{{\mathbb{R}^m}}
\newcommand{\Rn}{{\mathbb{R}^n}}

\newcommand{\Rmn}{{\mathbb{R}^{m\times n}}}
\newcommand{\Rnn}{{\mathbb{R}^{n\times n}}}
\newcommand{\Rnr}{{\mathbb{R}^{n\times r}}}

\newcommand{\diag}{{\rm diag}}

\newcommand{\dist}{{\rm dist}}
\newcommand{\dom}{{\rm dom}}

\newcommand{\prox}{{\rm prox}}

\newcommand{\sign}{{\rm sign}}

\newcommand{\BE}{\begin{equation}}
\newcommand{\EE}{\end{equation}}

\DeclareMathOperator*{\argmin}{argmin}
\newcommand{\normmm}[1]{{\vert\kern-0.25ex \vert\kern-0.25ex \vert #1
    \vert\kern-0.25ex \vert\kern-0.25ex\vert}}

\begin{document}
\title{\bf A Geometric Proximal Gradient Method for Sparse Least Squares Regression with Probabilistic
Simplex Constraint}
\author{Guiyun Xiao\thanks{School of Mathematical Sciences, Xiamen University, Xiamen 361005, People's Republic of China  (xiaogy999@163.com).}
\and Zheng-Jian Bai\thanks{Corresponding author. School of Mathematical Sciences and Fujian Provincial Key Laboratory on Mathematical Modeling \& High Performance Scientific Computing,  Xiamen University, Xiamen 361005, People's Republic of China (zjbai@xmu.edu.cn). }
}
\maketitle
\begin{abstract}
In this paper, we consider the sparse least squares regression problem with probabilistic
simplex constraint. Due to the probabilistic simplex constraint, one could not apply the $\ell_1$ regularization to the considered regression model. To find a sparse solution, we reformulate the  least squares regression problem as a nonconvex and nonsmooth $\ell_1$ regularized  minimization problem over the unit sphere. Then we propose a geometric proximal gradient method for solving the regularized problem,  where the explicit expression of the global solution to every involved subproblem is obtained. The global convergence of the proposed method is established under some mild assumptions.  Some numerical results are reported  to  illustrate the effectiveness of the proposed algorithm.
\end{abstract}

\vspace{3mm}

\vspace{3mm}

\section{Introduction}\label{sec1}
In this paper, we focus on the solution of  the following least squares regression with probabilistic
simplex constraint:
\BE\label{intro}
\begin{array}{lc}
\min\limits_{\bx\in\Rn}  &  \displaystyle \frac{1}{2}\|A\bx-\bb\|^2 \\[2mm]
\mbox{subject to (s.t.)} &{\bf 1}_n^T\bx=1, \\[2mm]
&  \bx\ge {\bf 0}
 \end{array}
\EE
for some sparse $\bx\in\Rn$,
where $A\in\Rmn$, ${\bf 0}\in\Rn$ is a zero vector, ${\bf 1}_n\in\Rn$ is a vector of all ones, and $\bx\ge {\bf 0}$ means that $\bx$ is a entry-wise nonnegative vector. Such problem arises in various applications such as the construction of probabilistic Boolean networks \cite{CJ12}, nonparametric distribution estimation \cite[\S7.2]{BV04}, and sparse hyperspectral unmixing \cite{IB09}, etc.

There exist various numerical methods for solving the optimzation  model related to problem \eqref{intro}. For example, Iordache et al. \cite{IB09} introduced the Moore-Penrose pseudoinverse method, the orthogonal matching pursuit method, the iterative spectral mixture analysis method, and the $\ell_2$--$\ell_1$ sparse regression technique. Bioucas-Dias and Figueiredo \cite{BF10} presented the alternating direction method of multipliers (ADMM) algorithm  for solving the following minimization  problem:
\BE\label{pro:mp}
\min\limits_{\bx\in\Rn}   \frac{1}{2}\|A\bx-\bb\|^2 +\la \|\bx\|_1 +\chi_{\{1\}}({\bf 1}^T\bx)+\chi_{\R_+^n}(\bx),
\EE
where $\la>0$ is a regularized parameter, $\R_+^n$ denotes the nonnegative orthant of $\Rn$ ($\R_+=\R_+^1$), and $\chi_\cd$ is a characteristic function of  a set $\cd\subset\Rn$ defined by
\[
\chi_\cd(x)=\left\{
\begin{array}{cl}
0,         &   x\in\cd,\\
+\infty,  & \mbox{otherwise}.
\end{array} \right.
\]
Salehani et al. \cite{SG14} provided  the ADMM method for solving the following regularized constrained sparse regression:
\[
\begin{array}{lc}
\min\limits_{\bx\in\Rn}  &  \displaystyle \frac{1}{2}\|A\bx-\bb\|^2+\la\Omega(\delta,\bx) \\[2mm]
\mbox{s.t.} &{\bf 1}_n^T\bx=1, \quad \bx\ge {\bf 0},
 \end{array}
\]
where $\la>0$ is a regularized parameter and $\Omega(\delta,\bx):=\sum_{j=1}^{n}\frac{2}{\pi}\arctan(\frac{x_j}{\delta})$ with $0< \delta<1$ is an arctan function, which is used to approximate $\ell_0$ norm since $\lim_{\delta\to 0}\Omega(\delta,\bx)=\|\bx\|_0$.
In \cite{CC11,CJ12}, Chen et al. gave a (generalized) maximum entropy rate  method for solving problem \eqref{intro}.   In \cite{WW15}, Wen et al. presented
a projection-based gradient descent method for solving problem \eqref{intro}.  In \cite{DP19},  Deng et al. provided a partial proximal-type operator splitting method   for solving the $\ell_{1/2}$ regularization version of problem  \eqref{intro}:
\[
\begin{array}{lc}
\min\limits_{\bx\in\Rn}  &  \displaystyle \frac{1}{2}\|A\bx-\bb\|^2 +\lambda \|\bx\|_{1/2}^{1/2} \\
\mbox{s.t.} &\bfo_n^T \bx=1,\quad \bx\ge {\bf 0},
 \end{array}
\]
where $\la>0$ is a regularized parameter.

In this paper, we first reformulate the  regression model \eqref{intro} as an nonconvex minimization problem over the unit sphere. To find a sparse solution, we add the $\ell_1$-penalty to the nonconvex minimization problem. Then we introduce a geometric proximal gradient method for solving the regularized problem. This is motivated by the book due to Hastie et al. \cite{HT15} and the paper due to Bolte et al. \cite{BS14}. As noted in \cite{HT15}, the lasso or $\ell_1$ regularization is widely employed for learning the sparsity of the  regression parameters $\{x_i\}$ in  high-dimensional linear regression and inverse problems. In \cite{BS14}, Bolte et al.  presented a proximal alternating linearized minimization algorithm for solving nonconvex and nonsmooth optimization problems, where, in each iteration, one only need to apply the proximal forward-backward scheme to minimizing the sum of a smooth function with a nonsomooth one. However, in  each iteration of the proposed method, we need to use the proximal mapping to the nonconvex and nonsmooth minimization of the sum of a smooth function with the $\ell_1$ regularization item over the unit sphere, which gives a challenge since it is hard to simplify the proximal mapping as the projection onto a closed set. We shall analyze the special property of the involved  proximal mapping and give the explicit expression of a global minimizer of each involved subproblem. By using the  Kurdyka-{\L}ojasiewicz property defined in  \cite{BS14}, the global convergence of the proposed method is established and  we also show that each sequence generated by our method converges to a critical point of some regularized function under some mild assumptions.  Some numerical results are reported  to  illustrate the effectiveness of the proposed method.

Throughout this paper, we use the following notation. Let $\Rmn$ be the set of all $m\times n$ real matrices and $\Rn=\R^{n\times 1}$. Let $\Rn$ be equipped with the Euclidean inner product $\langle \cdot,\cdot\rangle$ and its induced norm $\|\cdot\|$. Let $\Rmn$ be equipped with the Frobenius inner product $\langle \cdot,\cdot\rangle_F$ and its induced Frobenius norm $\|\cdot\|_F$. The superscript ``$\cdot^T$" stands for the  transpose of a matrix or vector. $I_n$ denotes the identity matrix of order $n$. The symbol `$\odot$'  means the Hadamard product of  two vectors. Let $|\cdot|$ be the absolute value of a real number or the components of a real vector. Denote by $\|\cdot \|_0$ and $\|\cdot\|_2$ the number of nonzero entries of a vector or a matrix and the matrix $2$-norm, respectively. Denote by $\diag(C)$ a diagonal matrix with the same diagonal entries as a square matrix $C$. For any  matrix $C=(c_{ij})\in\Rmn$, let $\|C\|_1:=\sum_{i,j}|c_{ij}|$. For any  $\bc=(c_1,\ldots,c_n)^T\in\Rn$, let $\max(\bc,{\bf 0}):=(\max(c_1,0),\ldots,\max(c_n,0))^T$ and $\sign(\bc):=(\sign(c_1),\ldots,\sign(c_n))^T$, where $\sign(c_k)=1$ if $c_k>0$, $\sign(c_k)=-1$ if $c_k<0$ and $\sign(c_k)\in [-1,1]$ if $c_k=0$.

The rest of this paper is organized as follows. In Section \ref{sec2}, we reformulate the  sparse least squares regression problem with probabilistic simplex constraint as a $\ell_1$ regularized problem  over the unit sphere and present a geometric proximal gradient method for solving the regularized problem. In Section \ref{sec3}, we derive the explicit expression of the global minimizer of each involved subproblem and estalish the global convergence
of the proposed method under some mild assumptions. In Section \ref{sec4}, we discuss some extensions. In Section \ref{sec5}, we report some numerical tests to indicate the effectiveness of our method. Finally, some
concluding remarks are given in Section \ref{sec6}.

\section{A geometric proximal gradient method} \label{sec2}
In this section, we first reformulate problem \eqref{intro} as a nonconvex and nonsmooth minimization problem over the unit sphere. To find a sparse solution, we use the $\ell_1$-penalty to the nonconvex and nonsmooth minimization problem. Then we propose a geometric proximal gradient method  for solving the regularized problem.
\subsection{Reformulation}
To find a sparse solution, sparked by \cite{HT15}, it is desired to directly add $\ell_1$-penalty to problem \eqref{intro}, i.e.,
\[
\begin{array}{lc}
\min\limits_{\bx\in\Rn}  &  \displaystyle \frac{1}{2}\|A\bx-\bb\|^2 +\lambda \|\bx\|_1 \\
\mbox{s.t.} &\bfo_n^T \bx=1,\quad \bx\ge {\bf 0},
 \end{array}
\]
where $\la>0$ is a regularized parameter. However, the $\ell_1$ norm of $\bx$ is a constant since $\|\bx\|_1={\bf 1}_n^T\bx=1$ for all possible points in the feasible set:
\[
\ck:=\{\bx\in\Rn\;| \: {\bf 1}_n^T\bx=1,\;  \bx\ge {\bf 0}\}.
\]

In the following, we reformulate  problem \eqref{intro} as a nonconvex and nonsmooth minimization problem over the unit sphere. It is easy to see that the feasible set $\ck$ of problem \eqref{intro} can be written as
\[
\ck=\{\by\odot\by\;| \: \by^T\by=1,\; \by\in\Rn\}.
\]
We note that the unit sphere $\cs^{n-1}:=\{\by\in\Rn\;| \: \by^T\by=1\}$ is a Riemannian submanifold of $\Rn$ . Then problem \eqref{intro} is reduced to the following  minimization problem over the unit sphere:
\BE\label{pbn}
\begin{array}{lc}
\min\limits_{\by\in\Rn}  &  \displaystyle \frac{1}{2}\|A(\by\odot \by)-\bb\|^2\\[2mm]
\mbox{s.t.} &\by\in\cs^{n-1}
 \end{array}
\EE
We note that if $\by^\#\in\cs^{n-1}$ is a solution to problem \eqref{pbn}, then $\bx^\#:=\by^\#\odot \by^\#\in\ck$ is a solution to problem \eqref{intro}.

To find a sparse solution to  problem \eqref{intro}, by adding the $\ell_1$-penalty to problem \eqref{pbn}, we get the following $\ell_1$ regularized problem:
\BE\label{pbn-r}
\begin{array}{lc}
\min\limits_{\by\in\Rn}  &  \displaystyle \frac{1}{2}\|A(\by\odot \by)-\bb\|^2 +\lambda \|\by\|_1 \\[2mm]
\mbox{s.t.} &\by\in\cs^{n-1},
 \end{array}
\EE
where $\la>0$ is a regularized parameter.

We point out that problem  \eqref{pbn-r} is a nonconvex and nonsmooth minimization problem over the Riemannian manifold $\cs^{n-1}$. For simplicity, we define
\BE\label{def:fg}
f(\by):=\frac{1}{2}\|A(\by\odot\by)-\bb\|^2,\quad
g(\lambda,\by):=\lambda \|\by\|_1,\quad
F(\lambda,\by):=f(\by)+g(\lambda,\by)+\chi_{\cs^{n-1}}(\by).
\EE
Using Definition \ref{app:lsc}, we can easily check that $\chi_{\cs^{n-1}}$ is a proper lsc function on $\Rn$.

\subsection{Geometric proximal gradient method}
In this subsection, we propose a geometric proximal gradient method with the varied regularization parameter $\la$ for solving problem \eqref{pbn-r}.  As in \cite[p. 20]{RW98} and \cite{BS14}, for a proper lsc function $\phi:\Rn\to [-\infty, \infty]$ and a constant $\alpha>0$, the proximal mapping of $\phi$ is defined by
\[
\prox^{\phi}_t(\widehat{\by})=\argmin_{\by\in\Rn}\left\{\phi(\by)
+\frac{1}{2\alpha}\|\by-\widehat{\by}\|^2\right\}.
\]
For example, for the characteristic function $\chi_\cd$ of  a set $\cd\subset\Rn$, the proximal mapping of $\chi_\cd$ is given by
\[
\prox^{\chi_\cd}_t(\widehat{\by})=\argmin_{\by\in\cd}\|\by-\widehat{\by}\|,
\]
i.e., the projection onto $\cd$. In particular, if $\cd=\cs^{n-1}$, then the projection onto $\cs^{n-1}$ has a unique solution if $\widehat{\by}\neq {\bf 0}$ \cite{AM12}. To apply the proximal gradient method (see for instance \cite{B17}) to problem \eqref{pbn-r}, we need to use the proximal mapping to minimizing the sum of the linearization of the smooth function $f$ at the current iterate $\by^k$, the nonsmooth function $g(\la,\cdot)$, and the characteristic function $\chi_{\cs^{n-1}}$, i.e., the next iterate $\by^{k+1}$ is defined by
\[
\by^{k+1}=\argmin_{\by\in\cs^{n-1}}\left\{ f(\by^k)+\big\langle\by-\by^k , \nabla f(\by^k) \big\rangle +\la\|\by\|_1+\frac{1}{2\alpha}\|\by-\by^k\|^2
\right\},
\]
where $\alpha$ and $\la$  are two positive constants.

Based on the above analysis, our geometric proximal gradient algorithm for solving  problem \eqref{pbn-r}
can be described as follows.
\begin{algorithm}[!h]
\caption{Geometric proximal gradient method with varied regularized parameter} \label{pgm}
\begin{description}
\item [{\rm Step 0.}] Choose $\by^0\in\cs^{n-1}$, $\alpha_0>0$, $\lambda_0>0$, $\rho_1,\rho_2,\rho_3\in(0,1)$, $\gamma_1,\gamma_2>0$,  $\delta_1>0$, $\delta_2>0$. Let $k:=0$.
\item [{\rm Step 1.}] Take $\alpha=\alpha_0$ and compute
\BE\label{ques1}
\overline{\by}^{k}=\argmin_{\by\in\cs^{n-1}} \Big\{f(\by^k)+\big\langle\by-\by^k , \nabla f(\by^k) \big\rangle +\frac{1}{2\alpha}\|\by-\by^k\|^2+{\lambda_k}\|\by\|_1\Big\}.
\EE
\item  [{\rm Step 2.}]
{\bf Repeat} until $F(\la_k,\overline{\by}^{k})\le F(\la_k,\by^k) - \frac{1}{2} \gamma_2\|\overline{\by}^{k}-\by^k\|^2$

$\qquad$ Set $\alpha=\max\{\gamma_1,\alpha\rho_1\}$.

$\qquad$ if $F(\la_k,\overline{\by}^{k})> \delta_1 F(\la_k,\by^k)$, then set $\alpha=\max\{\gamma_1,\alpha\rho_2\}$.

$\qquad$ if $|F(\la_k,\overline{\by}^{k})-  F(\la_k,\by^k)| < \delta_2 F(\la_k,\by^k)$, then replace $\la_k$ by $\la_k\rho_3$.

$\qquad$ Compute
\[
\overline{\by}^{k}=\argmin_{\by\in\cs^{n-1}} \Big\{f(\by^k)+\big\langle\by-\by^k , \nabla f(\by^k) \big\rangle +\frac{1}{2\alpha}\|\by-\by^k\|^2+{\lambda_k }\|\by\|_1\Big\}.
\]
$\quad$ {\bf end (Repaeat)}

$\quad$ Set $\by^{k+1}:=\overline{\by}^{k}$, $\alpha_{k+1}:=\alpha$, and $\la_{k+1}:=\la_k$.

\item [{\rm Step 3.}]  Replace $k$ by $k+1$ and go to  Step 1.
\end{description}
\end{algorithm}

On Algorithm \ref{pgm}, we have the following remarks.
\begin{remark}
From the proof of Lemma \ref{lem:mon} below, we observe that the condition $F(\la_k,\overline{\by}^{k})\le F(\la_k,\by^k) - \frac{1}{2} \gamma_2\|\overline{\by}^{k}-\by^k\|^2$ holds if we fix $\alpha_{k+1}$ such that $\alpha_{k+1} \le 1/(L_f+\gamma_2)$, where $L_f$ is the Lipschitz constant as given in Lemma {\rm\ref{f:lc}}. However, this upper bound may be very small numerically, which is not necessary in practice. Therefore, in  Algorithm \ref{pgm}, we find a search stepsize $\alpha_{k+1}$ starting from a resonable large $\alpha_0>0$.
\end{remark}
\begin{remark}
In  Algorithm \ref{pgm}, we choose a regularization parameter $\la$ such that the relative error $|F(\la_{k+1},\by^{k+1})-  F(\la_k,\by^k)|/F(\la_k,\by^k)$ is not too small. From the latter numerical tests, we see that such adjustment of the regularization parameter may achieve a good  tradeoff between  sparsity and objective function value.
\end{remark}
%
%
%
%
\section{Convergence analysis}\label{sec3}
In this section, we discuss the global convergence of Algorithm \ref{pgm}. We first derive the explicit expression of the global minimizer of  the nonconvex and nonsmooth minimization  problem as defined in \eqref{ques1}. Then we show the propose method is globally convergent and the sequence $\{\by^k\}$ generated by Algorithm \ref{pgm} converges to a critical point of  problem  \eqref{pbn-r} with some regularization parameter under the assumption that  $\la_k$ is fixed for all $k$ sufficiently large.

\subsection{Explicit expression of $\overline{\by}^{k}$ defined in \eqref{ques1}}
In this subsection, we derive the explicit expression of $\overline{\by}^{k}\in\cs^{n-1}$ defined in \eqref{ques1}, which aims to solve the following minimization problem: For any given $\bfd\in\cs^{n-1}$ and two constants  $\alpha,\la>0$, compute
\BE\label{ques}
\begin{array}{lc}
\min\limits_{\by\in\Rn}  &  \displaystyle \big\langle\by-\bfd , \nabla f(\bfd) \big\rangle +\frac{1}{2\alpha}\|\by-\bfd\|^2+{\lambda}\|\by\|_1\\[2mm]
\mbox{s.t.} &\by\in\cs^{n-1}.
 \end{array}
\EE

To find a global solution to problem \eqref{ques}, we first derive the following necessary condition.
\begin{lemma} \label{lem:xkadd}
If $\overline{\by}\in\cs^{n-1}$ solves problem \eqref{ques}, then we have
\[
 \overline{\by} \odot \bz \ge {\bf 0},
\]
where $\bz:=\bfd-{\alpha} \nabla f(\bfd)$.
\end{lemma}

\begin{proof}
Let $\overline{\by}:=(\overline{y}_1,\ldots,\overline{y}_n)^T$ and $\bz:=(z_1,\ldots,z_n)^T$.
For the sake of contradiction, suppose there exists an index $1\le t\le n$ such that $\overline{y}_tz_t<0$. It follows that $\overline{\by}\in\cs^{n-1}$ is a solution to problem \eqref{ques} if and only if
\BE \label{def:xk1}
\overline{\by}=\argmin_{\by\in\cs^{n-1}} \Big\{\frac{1}{2\alpha}\| \by-\bz\|^2+\lambda\|\by\|_1\Big\}.
\EE
Let $\widetilde{\by}=(\overline{y}_1, \ldots, \overline{y}_{t-1}, -\overline{y}_t, \overline{y}_{t+1}, \ldots,\overline{y}_n)^T$. Then it is easy to verify that $\widetilde{\by}\in \cs^{n-1}$ and
\[
\Big(\frac{1}{2\alpha}\|\overline{\by}-\bz\|^2+\lambda\|\overline{\by}\|_1\Big)
- \Big(\frac{1}{2\alpha}\|\widetilde{\by}-\bz\|^2+\lambda\|\widetilde{\by}\|_1\Big)=-\frac{2}{\alpha} \overline{y}_tz_t>0.
\]
This contradicts the assumption that $\overline{\by}$ solves \eqref{def:xk1}. Therefore, we have $\overline{\by}\odot \bz \ge {\bf 0}$.
\end{proof}

On the explicit expression of a global solution to problem \eqref{ques}, we have the following theorem.
\begin{theorem} \label{thm:xksol}
Let $\bz=\bfd-\alpha \nabla f(\bfd)\equiv (z_1,\ldots, z_n)^T$. Define $\bv=(v_1,\ldots, v_n)^T\in\Rn$ and $\bw=(w_1,\ldots, w_n)^T\in\Rn$ by
\BE\label{def:zwk}
v_j=\left\{
\begin{array}{cl}
1, &\mbox{if $z_j\ge 0$},\\[2mm]
-1, & \mbox{otherwise}
\end{array}
\right.\quad\mbox{and}\quad
w_j=\lambda-\frac{1}{\alpha}|z_j|,
\EE
for $j=1,\ldots,n$. Let $t=\argmin_{j\in[n]}w_j$ and $\bw_{-}=\min(0,\bw)$.
Then a global solution  to  \eqref{ques} is given by
\[
\overline{\by}=\left\{
\begin{array}{ll}
(\underbrace{0, \ldots, 0}_{t-1}, v_t, \underbrace{0, \ldots, 0}_{n-t})^T, &\mbox{if $\bw\ge {\bf 0}$},\\[2mm]
-\frac{\bw_{-}}{\|\bw_{-}\|}\odot \bv, & \mbox{otherwise}.
\end{array}
\right.
\]
\end{theorem}
\begin{proof}
Using Lemma \ref{lem:xkadd} and \eqref{def:xk1} we have
\BE\label{xk:exp1}
\overline{\by}=\argmin_{\by\in\cs^{n-1}_*} \Big\{-\frac{1}{\alpha}\by^T\bz+\lambda\|\by\|_1\Big\},
\EE
where $\cs^{n-1}_*:=\{\by\in\cs^{n-1}\; | \; \by \odot \bz \ge {\bf 0}\}$. For any $\by\in\cs^{n-1}_*$, we have
\begin{eqnarray} \label{xk:exp2}
-\frac{1}{\alpha}\by^T\bz+\lambda\|\by\|_1  &=& \sum\limits_{j=1}^n\big(-\frac{1}{\alpha}y_jz_j+\lambda|y_j|\big)\nonumber \\
&=&\sum\limits_{j=1}^n\big(-\frac{1}{\alpha}|y_j| |z_j|+\lambda|y_j|\big) \nonumber \\
&=&\sum\limits_{j=1}^n\big(\la-\frac{1}{\alpha}| z_j|\big) |y_j| \equiv\sum\limits_{j=1}^nw_j |y_j|,
\end{eqnarray}
where $w_j$'s are defined by \eqref{def:zwk}.

We now determine a global solution $\overline{\by}\in\cs^{n-1}$ to problem  \eqref{ques} as follows: We first assume that $\bw\geq 0$. Let
\BE\label{zstar:1}
\overline{\by}=(\underbrace{0, \ldots, 0}_{t-1}, v_{t}, \underbrace{0, \ldots, 0}_{n-t})^T \equiv(\overline{y}_1,\ldots,\overline{y}_n)^T\in\cs^{n-1}_*,
\EE
where $\bv$ is  defined by \eqref{def:zwk} and $t=\argmin_{j\in[n]}w_j$. Then, for any $\by\in\cs^{n-1}_*$,  we have
\begin{eqnarray*}
\sum\limits_{j=1}^nw_j |y_j| \ge w_t\sum\limits_{j=1}^n|y_j|\geq w_t\sum\limits_{j=1}^ny_j^2=w_t
=\sum\limits_{j=1}^n w_j|\overline{y}_j|.
\end{eqnarray*}
Hence, it follows from \eqref{xk:exp1} and \eqref{xk:exp2} that $\overline{\by}$ defined in \eqref{zstar:1} solves problem \eqref{ques}.

On the other hand, suppose there exists at least one index $ l\in [n]$ such that $w_l<0$. Let $\bw_{-}=\min(0,\bw)\neq {\bf 0}$ and
\BE\label{zstar:2}
\overline{\by}=-\frac{\bw_{-}}{\|\bw_{-}\|}\odot \bv\equiv(\overline{y}_1,\ldots,\overline{y}_n)^T\in\cs^{n-1}_*,
\EE
where $\bv$ is  defined by \eqref{def:zwk}.
Then, for any $\by\in\cs^{n-1}_*$,  we have
\begin{eqnarray*}
\sum\limits_{j=1}^nw_j|y_j|\geq\sum\limits_{w_j<0}w_j|y_j|=\langle \bw_{-}, |\by|\rangle \geq-\|\bw_{-}\|\|\by\|=-\|\bw_{-}\|=\sum\limits_{j=1}^nw_j|\overline{y}_j|.
\end{eqnarray*}
This, together with \eqref{xk:exp1} and \eqref{xk:exp2}, implies  that  $\overline{\by}$ defined by \eqref{zstar:2} solves problem \eqref{ques}.
\end{proof}

Based on Theorem \ref{thm:xksol}, we can easily find the explicit expression of $\overline{\by}^{k}\in\cs^{n-1}$ defined in \eqref{ques1}, which is stated as Algorithm \ref{dm}.
\begin{algorithm}[!h]
\caption{Computing a global solution to  (\ref{ques1})} \label{dm}
\begin{description}
\item [{\rm Step 0.}] Compute $\bz^k=\by^k-{\alpha} \nabla f(\by^k)\equiv (z_1^k,\ldots, z_n^k)^T$. Set $\bv^k=(v_1^k,\ldots, v_n^k)^T\in\Rn$ and $\bw^k=(w_1^k,\ldots, w_n^k)^T\in\Rn$ with
\[
v_j^k=\left\{
\begin{array}{rl}
1, &\mbox{if $z_j^k\ge {\bf 0}$},\\[2mm]
-1, & \mbox{otherwise}
\end{array}
\right.\quad\mbox{and}\quad
w_j^k=\lambda^k-\frac{1}{\alpha}|z_j^k|, \quad j=1,\ldots,n.
\]

\item [{\rm Step 1.}] find $j_k=\argmin_{j\in[n]} w^k_j$ and set $\bw^k_{-}:=\min(0,\bw^k)$.

\item [{\rm Step 2.}] If $w^k_{j_k}\ge {\bf 0}$, then set $\overline{\by}^{k}=(\underbrace{0, \ldots, 0}_{j_k-1}, v_{j_k}^k, \underbrace{0, \ldots, 0}_{n-j_k})^T$; otherwise, set $\overline{\by}^{k}=-\frac{\bw^k_{-}}{\|\bw^k_{-}\|}\odot \bv^k$.
\end{description}
\end{algorithm}
\subsection{Lipschitz continuity of the gradient mapping $\nabla f$}\label{sec32}

It is easy to verify that the gradient of $f$ defined in \eqref{def:fg} at $\by\in\cs^{n-1}$ is given by
\BE\label{f:grad}
\nabla f(\by)=2\big(A^TA(\by\odot \by)-A^T\bb\big)\odot\by.
\EE

On the global Lipschitz continuity of $\nabla f$ on the closed unit ball in $\Rn$, we have the following result.
\begin{lemma}\label{f:lc}
Let $\cb_1({\bf 0}):=\{\by\in\Rn\; |\; \|\by\|\le 1\}$ be the closed unit ball in $\Rn$. Then there exists a constant $L_f>0$ such that
\[
\|\nabla f(\by_1)-\nabla f(\by_2)\|\le L_f\|\by_1-\by_2\|,\quad  \forall \by_1,\by_2\in\cb_1({\bf 0}).
\]
\end{lemma}
\begin{proof}
It follows from \eqref{f:grad} that, for any $\by_1,\by_2\in\cb_1({\bf 0})$,
\begin{eqnarray*}
&&\left\|\nabla f(\by_1)-\nabla f(\by_2)\right\|\nonumber \\
&=& 2\left\|\big(A^TA(\by_1\odot \by_1)-A^T\bb\big)\odot\by_1 - \big(A^TA(\by_2\odot \by_2)-A^T\bb\big)\odot\by_2\right\| \nonumber \\
&\le& 2\left\|\big(A^TA(\by_1\odot \by_1)\big)\odot\by_1-\big(A^TA(\by_2\odot \by_2)\big)\odot\by_2\right\|
+ 2\left\|(A^T\bb)\odot(\by_1 - \by_2)\right\| \nonumber \\
&\le& 2\left\|\big(A^TA(\by_1\odot \by_1)\big)\odot(\by_1-\by_2)\right\| +
2\left\|\big(A^TA(\by_1\odot \by_1 - \by_2\odot \by_2)\big)\odot\by_2\right\| \nonumber \\
&& + 2\left\|(A^T\bb)\odot(\by_1 - \by_2)\right\|\nonumber\\
&\le& 2\left\|\big(A^TA(\by_1\odot \by_1)\big)\odot(\by_1-\by_2)\right\| +
2\left\|\big(A^TA(\by_1\odot (\by_1 -\by_2))\big)\odot\by_2\right\| \nonumber
\end{eqnarray*}
\begin{eqnarray*}
&& + 2\left\|\big(A^TA((\by_1 -\by_2)\odot\by_2)\big)\odot\by_2 \right\| + 2\left\|(A^T\bb)\odot(\by_1 - \by_2)\right\|\nonumber\\
&\le& 2\left\|A^TA(\by_1\odot \by_1)\right\|\left\|\by_1-\by_2\right\| +
2\left\|A^TA(\by_1\odot (\by_1 -\by_2))\right\|\left\|\by_2\right\| \nonumber \\
&& + 2\left\|A^TA((\by_1 -\by_2)\odot\by_2)\right\|\left\|\by_2 \right\| + 2\left\|A^T\bb\right\| \left\|\by_1 - \by_2\right\|\nonumber\\
&\le& 2\left\|A^TA\right\|_2 \left\|\by_1\odot \by_1\right\|\left\|\by_1-\by_2\right\| +
2 \left\|A^TA\right\|_2\|\by_1\|\left\|\by_1 -\by_2\right\| \|\by_2\|\nonumber \\
&& + 2\left\|A^TA\right\|_2\left\|\by_1 -\by_2\right\| \|\by_2\|^2 + 2\left\|A^T\bb\right\| \left\|\by_1 - \by_2\right\|\nonumber \\
&\le& L_f\left\|\by_1 - \by_2\right\|,
\end{eqnarray*}
where the fourth inequality follows from the fact that $\|\bz_1\odot \bz_2\|\le \|\bz_1\|\|\bz_2\|$ for all $\bz_1,\bz_2\in\Rn$, the fifth and sixth inequalities use $\|\by_1\odot \by_1\|\le \|\by_1\|\le 1$ and $\|\by_2\|\le 1$ for all $\by_1,\by_2\in\cb_1({\bf 0})$, and $L_f=6\|A^TA\|_2 + 2\|A^T\bb\|$.
\end{proof}

We also recall the descent lemma for the continuously differentiable function $f$  defined in \eqref{def:fg} (see for instance \cite{B99,OR70}).
\begin{lemma}\label{lem:dl}
Under the same assumptions as in Lemma {\rm\ref{f:lc}}, we have
\[
f(\by_2)\le f(\by_1)+\langle\by_2-\by_1,\nabla f(\by_1)\rangle + \frac{L_f}{2}\|\by_2-\by_1\|^2,\quad\forall \by_1,\by_2\in\cb_1({\bf 0}).
\]
\end{lemma}

\subsection{Global convergence of Algorithm \ref{pgm}} \label{sec33}
In this subsection, we establish the  global convergence of Algorithm \ref{pgm}. We first derive the monotonicity of the sequences  $\{\la_k\}$ and $\{F(\la_k,\by^k)\}$ generated by Algorithm \ref{pgm} in a similar way as \cite[Lemma 3]{BS14}.

\begin{lemma}\label{lem:mon}
Let $\{(\by^k,\alpha_k,\la_k)\}$  be the sequence generated by Algorithm \ref{pgm}. Then the following conclusions hold true.
\begin{itemize}
  \item[{\rm i)}] The sequence  $\{\la_k\}$ is monotonically decreasing, which converges to a a limit $\la_*$.
  \item [{\rm ii)}] We can find $\alpha_k>0$  such that the sequence $\{F(\la_k,\by^k)\}$ is monotonically decreasing and for all $k\ge {\bf 0}$,
\[
F(\la_{k},\by^{k}) - F(\la_{k+1},\by^{k+1})\geq\frac{\gamma_2}{2}\|\by^{k+1}-\by^k\|^2.
\]
  \item [{\rm iii)}]
  \[
\sum_{k=0}^\infty\|\by^{k+1}-\by^k\|^2<\infty.
\]
\end{itemize}

\end{lemma}
\begin{proof}
i) From  Algorithm \ref{pgm}, it is easy to see that the sequence $\{\la_k\}$ is monotonically decreasing and bounded below. Therefore, it converges to a limit $\la_*$.

ii)
Let $k\ge 0$ be fixed. It follows from Step 1 and Step 2 Algorithm \ref{pgm} that
\[
\by^{k+1}=\argmin_{\by\in\cs^{n-1}} \big\{f(\by^k)+\langle\by-\by^k , \nabla f(\by^k) \rangle +\frac{1}{2\alpha_{k+1}}\|\by-\by^k\|^2+\la_{k+1}\|\by\|_1\big\},
\]
which, together with the fact that $\by^{k},\by^{k+1}\in \cs^{n-1}$, and $\la_{k+1}\le \la_{k}$, yields
\BE\label{eq:s12}
\langle\by^{k+1}-\by^k , \nabla f(\by^k) \rangle +\frac{1}{2\alpha_{k+1}}\|\by^{k+1}-\by^k\|^2+\la_{k+1}\|\by^{k+1}\|_1\leq \la_{k+1}\|\by^k\|_1\leq \la_{k}\|\by^k\|_1.
\EE
Using Lemma \ref{lem:dl} and \eqref{eq:s12} we have
\begin{eqnarray*}
&& f(\by^{k+1})+\la_{k+1}\|\by^{k+1}\|_1 \\
&\le& f(\by^{k+1}) -\langle\by^{k+1}-\by^k , \nabla f(\by^k) \rangle - \frac{1}{2\alpha_{k+1}}\|\by^{k+1}-\by^k\|^2 + \la_{k}\|\by^{k}\|_1\\
&=& \left(f(\by^k)+\la_{k}\|\by^k\|_1\right) - \frac{1}{2\alpha_{k+1}}\|\by^{k+1}-\by^k\|^2 \\
&& + \left(f(\by^{k+1}) -  f(\by^k) -\langle\by^{k+1}-\by^k , \nabla f(\by^k) \rangle \right)\\
&\le& \left(f(\by^k)+\la_{k}\|\by^k\|_1\right) -\frac{1}{2}\left(\frac{1}{\alpha_{k+1}}- L_f\right)\|\by^{k+1}-\by^k\|^2.
\end{eqnarray*}
Taking $\alpha_{k+1} \le 1/(L_f+\gamma_2)$ we have
\[
\left(f(\by^{k})+\la_{k+1}\|\by^{k}\|_1\right)-\left(f(\by^{k+1})+\la_{k+1}\|\by^{k+1}\|_1\right)\geq\frac{\gamma_2}{2}\|\by^{k+1}-\by^k\|^2.
\]
This implies that
\BE\label{eq:fk1}
F(\la_{k},\by^{k}) - F(\la_{k+1},\by^{k+1})\geq\frac{\gamma_2}{2}\|\by^{k+1}-\by^k\|^2,\quad \forall k\ge 0.
\EE
This shows that the sequence $\{F_k(\by^k)\}$  is  monotonically decreasing.

iii) It follows from  \eqref{eq:fk1} that, for any integer $l>0$,
\[
\sum_{k=0}^{l}\|\by^{k+1}-\by^k\|^2 \le \frac{2}{\gamma_2} \left( F(\la_{0},\by^{0}) - F(\la_{l+1},\by^{l+1})\right) \le\frac{2}{\gamma_2} F(\la_{0},\by^{0})
\]
and thus
\[
\sum_{k=0}^{\infty}\|\by^{k+1}-\by^k\|^2 <\infty.
\]
\end{proof}

\begin{lemma}\label{lem:grad}
Let $\{(\by^k,\alpha_k,\la_k)\}$  be the sequence generated by Algorithm \ref{pgm}. For any $k\ge 0$, let
\[
\bq^k:= \nabla f(\by^{k+1}) - \nabla f(\by^k) -\frac{1}{\alpha_{k+1}}(\by^{k+1}-\by^{k}).
\]
Then there exist $\eta^{k+1}\in \partial_\by (\|\by^{k+1}\|_1)$ and $\zeta^{k+1}\in\hat{\partial}\chi_{\cs^{n-1}}(\by^{k+1})$ such that  $\bq^k=\nabla f(\by^{k+1})+\la_{k+1}\eta^{k+1}+\zeta^{k+1}\in \partial_\by  F(\la_{k+1},\by^{k+1})$ and we have for all $k\ge 0$,
\[
\|\bq^k\|\le \left(L_f+\frac{1}{\gamma_1}\right) \|\by^{k+1}-\by^{k}\|,
\]
where $L_f$ is a constant defined as in Lemma \ref{f:lc}.
\end{lemma}
\begin{proof}
Let $k\ge 0$ be fixed.
From Step 1 and Step 2 Algorithm \ref{pgm} we obtain $\alpha_{k+1}\ge \gamma_1>0$ and
\[
\by^{k+1}=\argmin_{\by\in\Rn} \big\{f(\by^k)+\langle\by-\by^k , \nabla f(\by^k) \rangle +\frac{1}{2\alpha_{k+1}}\|\by-\by^k\|^2+\la_{k+1}\|\by\|_1+\chi_{\cs^{n-1}}(\by)\big\},
\]
whose first-order optimality condition is given by
\[
\nabla f(\by^k)+\frac{1}{\alpha_{k+1}}(\by^{k+1}-\by^{k})+\la_{k+1}\eta^{k+1}+\zeta^{k+1}=0,
\]
where $\eta^{k+1}\in \partial_\by (\|\by^{k+1}\|_1)$ and $\zeta^{k+1}\in\hat{\partial}\chi_{\cs^{n-1}}(\by^{k+1})$.
Therefore,
\[
\nabla f(\by^k)+\la_{k+1}\eta^{k+1}+\zeta^{k+1}=-\frac{1}{\alpha_{k+1}}(\by^{k+1}-\by^{k})
\]
and thus
\[
\nabla f(\by^{k+1})+\la_{k+1}\eta^{k+1}+\zeta^{k+1}= \nabla f(\by^{k+1}) - \nabla f(\by^k) -\frac{1}{\alpha_{k+1}}(\by^{k+1}-\by^{k})\equiv \bq^k.
\]
It is easy to see that
\[
\nabla f(\by^{k+1})+\la_{k+1}\eta^{k+1}+\zeta^{k+1}\in \partial_\by  F(\la_{k+1},\by^{k+1}).
\]
Hence, $\bq^k\in \partial_\by  F(\la_{k+1},\by^{k+1})$.

On the other hand, we note that $\by^k,\by^{k+1}\in \cs^{n-1}$. Using the definition of $\bq^k$ and Lemma \ref{f:lc} we obtain, for all $k\ge 0$,
\begin{eqnarray*}
\|\bq^k\|&\leq&\|\nabla f(\by^{k+1})-\nabla f(\by^k)\|+ \frac{1}{\alpha_{k+1}}\|\by^{k+1}-\by^{k}\| \\
&\le& \left(L_f+\frac{1}{\alpha_{k+1}}\right) \|\by^{k+1}-\by^{k}\| \le \left(L_f+\frac{1}{\gamma_1}\right) \|\by^{k+1}-\by^{k}\|.
\end{eqnarray*}
\end{proof}

On the global convergence of Algorithm \ref{pgm}, we have the following result.
\begin{theorem}\label{thm:gc}
Let $\{(\by^k,\alpha_k,\la_k)\}$  be the sequence generated by Algorithm \ref{pgm} with $\lim_{k\to\infty}\lambda_k$ $=\lambda_*$. Then any accumulation point of  $\{\by^k\}$  is a critical point of $F(\la_*,\cdot)$.
\end{theorem}
\begin{proof}
Let $\by^*\in\cs^{n-1}$ be an accumulation point of the sequence $\{\by^k\}$. Then there exists a subsequence  $\{\by^{k_t}\}$ converging to  $\by^*$. We note that $\chi_{\cs^{n-1}}$ is lsc. Thus,
\BE\label{eq:phi-inf}
\liminf_{t\to\infty}\chi_{\cs^{n-1}}(\by^{k_t})\ge\chi_{\cs^{n-1}}(\by^*).
\EE
From Step 1 and Step 2 Algorithm \ref{pgm} we have
\[
\by^{k+1}=\argmin_{\by\in\Rn} \big\{f(\by^k)+\langle\by-\by^k , \nabla f(\by^k) \rangle +\frac{1}{2\alpha_{k+1}}\|\by-\by^k\|^2+\la_{k+1}\|\by\|_1+\chi_{\cs^{n-1}}(\by)\big\}
\]
and thus
\begin{eqnarray}\label{eq:ykystar}
&& \langle\by^{k+1}-\by^k , \nabla f(\by^k) \rangle +\frac{1}{2\alpha_{k+1}}\|\by^{k+1}-\by^k\|^2+\la_{k+1}\|\by^{k+1}\|_1+\chi_{\cs^{n-1}}(\by^{k+1}) \nonumber\\
&\le&  \langle\by^*-\by^k , \nabla f(\by^k) \rangle +\frac{1}{2\alpha_{k+1}}\|\by^*-\by^k\|^2+\la_{k+1}\|\by^*\|_1+\chi_{\cs^{n-1}}(\by^*).
\end{eqnarray}
By hypothesis, $\lim_{k\to\infty}\lambda_k=\lambda_*$.  By Algorithm \ref{pgm}, we know that $\alpha_{k+1}\in [\gamma_1, \alpha_0]$ for all $k\ge 0$ and $\{\by^k\}$ is bounded. Using Lemma \ref{lem:mon} iii) we have
\[
\lim\limits_{k\to\infty}\|\by^{k+1}-\by^k\|=0.
\]
Hence, taking $k=k_t$ in \eqref{eq:ykystar} and let $t\to\infty$ and using the continuity of $f$ we obtain
\[
\limsup_{t\to\infty}\chi_{\cs^{n-1}}(\by^{k_t}) \le\chi_{\cs^{n-1}}(\by^*).
\]
This, together with \eqref{eq:phi-inf}, implies that
\[
\lim_{t\to\infty}\chi_{\cs^{n-1}}(\by^{k_t}) =\chi_{\cs^{n-1}}(\by^*).
\]
Therefore,
\begin{eqnarray}\label{eq:F-lim}
\lim_{t\to\infty}F(\la_{k_t},\by^{k_t}) &=& \lim_{t\to\infty}\left( f(\by^{k_t}) +\la_{k_t}\|\by^{k_t}\|_1+\chi_{\cs^{n-1}}(\by^{k_t}) \right) \nonumber\\
&=& f(\by^*) +\la_*\|\by^*\|_1+\chi_{\cs^{n-1}}(\by^*) = F(\la_*,\by^*).
\end{eqnarray}

On the other hand, using Lemma \ref{lem:mon} iii) we have
\[
\lim\limits_{k\to\infty}\|\by^{k+1}-\by^k\|=0.
\]
By Lemma \ref{lem:grad} we have $\bq^k\in \partial_\by  F(\la_{k+1},\by^{k+1})$ and
\[
\|\bq^k\| \le\left(L_f+\frac{1}{\gamma_1}\right) \|\by^{k+1}-\by^{k}\|
\]
and thus
\[
\lim_{k\to\infty}\bq^k={\bf 0}.
\]
Since $\partial_\by F(\cdot)$ is closed, we know that ${\bf 0}\in \partial_\by F(\la_*,\by^*)$, i.e., $\by^*$ is a critical point of $F(\la_*,\cdot)$.
\end{proof}

In the rest of this section, we show the the sequence $\{\by^k\}$ generated by Algorithm \ref{pgm} converges under some assumptions. We first give some necessary results.

Let $\cl(\by^0)$ be the set of all accumulation points of the sequence $\{\by^k\}$ generated by Algorithm \ref{pgm}, i.e.,
\[
\cl(\by^0)=\{\by^*\in\cs^{n-1}\;| \; \mbox{there exists a subsequence $\{\by^{k_t}\}$ such that $\lim_{t\to\infty}\by^{k_t}=\by^*$}\}.
\]

On the set $\cl(\by^0)$, we have the following result.
\begin{lemma}\label{lem:ly0}
Let $\{(\by^k,\alpha_k,\la_k)\}$  be the sequence generated by Algorithm \ref{pgm} with $\lim_{k\to\infty}\lambda_k$ $=\lambda_*$. Then
$\cl(\by^0)$ is a nonempty and compact set and the function $F(\la_*,\cdot)$ is finite and constant on $\cl(\by^0)$.
\end{lemma}
\begin{proof}
It is obvious that  $\cl(\by^0)$ is nonempty and compact since the sequence $\{\by^k\}$ is bounded. By Lemma \ref{lem:mon} ii) we know that the sequence $\{F(\la_k,\by^k)\}$ is monotonically decreasing and bounded below. Hence, the sequence $\{F(\la_k,\by^k)\}$ converges to a limit $F_*$, i.e., $\lim_{k\to\infty}F(\la_k,\by^k) =F_*$. By the definition of $F$ as in \eqref{def:fg}  we have
\BE\label{eq:F-lay}
F(\la_*,\by^k)=F(\la_k,\by^k)-(\lambda_k-\lambda_*)\|\by^k\|_1.
\EE
By Lemma \ref{lem:mon} i) we have $\lim_{k\to\infty}\lambda_k=\lambda_*$. Since the sequence $\{\by^k\}$ is bounded, it follows from \eqref{eq:F-lay} that
\[
\lim_{k\to\infty}F(\la_*,\by^k)=F_*.
\]
For any $\by^*\in\cl(\by^0)$, there exists a subsequence $\{\by^{k_t}\}$ such that $\lim_{t\to\infty}\by^{k_t}=\by^*$. By using the similar proof of \eqref{eq:F-lim} we obtain
\[
F(\la_*,\by^*)=F_*
\]
Hence, the function $F(\la_*,\cdot)$ is finite and constant on $\cl(\by^0)$.
\end{proof}

On the convergence of the sequence $\{\by^k\}$ generated by Algorithm \ref{pgm}, we have the following theorem, whose proof is similar to \cite[Theorem 1]{BS14}. We give the proof here for the sake of completeness.
\begin{theorem}
Let $\{(\by^k,\alpha_k,\la_k)\}$  be the sequence generated by Algorithm \ref{pgm}. If $\la_k=\la_*$ for all $k$ sufficiently large, then the sequence $\{\by^k\}$ converges to a critical point of $F(\la_*,\cdot)$.
\end{theorem}
\begin{proof}
From Algorithm \ref{pgm} we observe that $\by^k\in\cs^{n-1}$ for all $k\ge 0$ and thus the sequence $\{\by^k\}$ is bounded. Let $\by^*$ be an accumulation point of $\{\by^k\}$. Then, there exists a subsequence $\{\by^{k_t}\}$ such that $\lim_{t\to\infty}\by^{k_t}=\by^*$. By  Theorem \ref{thm:gc}, we know that $\by^*$ is  a critical point of $F(\la_*,\cdot)$. Following the similar proof of the first part of Theorem \ref{thm:gc} we have
\begin{eqnarray}\label{eq:F-lim2}
\lim_{t\to\infty}F(\la_{k_t},\by^{k_t}) = F(\la_*,\by^*).
\end{eqnarray}
If there exists an integer $\hat{k}>0$ such that $F(\la_{\hat{k}},\by^{\hat{k}}) = F(\la_*,\by^*)$, then it follows from Lemma \ref{lem:mon} ii) that $\by^{\hat{k}+1}=\by^{\hat{k}}$. By the induction, we can easily show that $\by^k=\by^{\hat{k}}$ for all $k\ge \hat{k}$ and thus $\lim_{k\to\infty}\by^{k}=\by^*$. Therefore, the conclusion holds.

We now suppose $F(\la_k,\by^k)\neq F(\la_*,\by^*)$ for all $k\ge 0$. By Lemma \ref{lem:mon} ii) we know that the sequence $\{F(\la_k,\by^k)\}$ is monotonically decreasing and thus  $F(\la_k,\by^k)> F(\la_*,\by^*)$ for all $k\ge 0$. For any $\eta>0$, it follows from \eqref{eq:F-lim2} that for all $k$ sufficiently large, $\lambda^k=\lambda_*$ and
\[
F(\la_*,\by^k)< F(\la_*,\by^*)+\eta
\]
It is obvious that $\lim_{k\to\infty}\dist(\by^k, \cl(\by^0))=0$. Therefore,  for any $\epsilon>0$, we have for all $k$ sufficiently large,
\[
\dist(\by^k, \cl(\by^0))<\epsilon.
\]
Thus, for all $k$ sufficiently large,
\[
\by^k\in\{\by\in\cs^{n-1}\; |\: \dist(\by,\cl(\by^0))<\epsilon\}\cap\{ F(\la_*,\by^*) <F(\la_*,\by)<F(\la_*,\by^*)+\eta\}.
\]
It follows from Lemma \ref{lem:ly0} that  $\cl(\by^0)$ is compact and $F(\la_*,\cdot)$ is constant on  $\cl(\by^0)$. Hence, using Lemma \ref{app:gkl} with $\cc=\cl(\by^0)$ and $h=F(\la_*,\cdot)$ we have, for all $k$ sufficiently large,
\[
\xi'\big(F(\la_*,\by^k)-F(\la_*,\by^*)\big)\; \dist(0, \partial_\by F(\la_*,\by^k))\geq 1,
\]
where  $\xi$ is a concave function defined as in Definition \ref{app:kli}.
By Lemma \ref{lem:grad} we have for all $k$ sufficiently large,
\[
\dist(0, \partial_\by F(\la_*,\by^{k}))\le\left(L_f+\frac{1}{\gamma_1}\right) \|\by^{k}-\by^{k-1}\|
\]
and thus, for all $k$ sufficiently large,
\BE\label{eq:phi-diff}
\xi'\big(F(\la_*,\by^k)-F(\la_*,\by^*)\big) \ge \frac{\gamma_1}{1+\gamma_1L_f}{\|\by^{k}-\by^{k-1}\|}^{-1}.
\EE
By  Lemma \ref{app:gkl}, we know that $\xi$ is concave. Then we have for all $k$ sufficiently large,
\begin{eqnarray}\label{eq:phi-cav}
&&\xi\big(F(\la_*,\by^k)-F(\la_*,\by^*) \big)-\xi\big(F(\la_*,\by^{k+1})-F(\la_*,\by^*)\big) \nonumber \\
&\ge& \xi'\big(F(\la_*,\by^k)-F(\la_*,\by^*) \big)\big(F(\la_*,\by^k)-F(\la_*,\by^{k+1})\big).
\end{eqnarray}
For arbitrary integers $r,s>0$, let
\[
\Theta_{r,s}:= \xi\big(F(\la_*,\by^r)-F(\la_*,\by^*)\big)- \xi\big(F(\la_*,\by^{s})-F(\la_*,\by^*)\big)
\quad\mbox{and}\quad \Psi:=\frac{2(1+\gamma_1L_f)}{\gamma_2\gamma_1}.
\]
From Lemma \ref{lem:mon} ii), \eqref{eq:phi-diff} and \eqref{eq:phi-cav} we have for all $k$ sufficiently large,
\[
\Theta_{k,k+1}\geq\frac{{\|\by^{k+1}-\by^{k}\|}^2}{\Psi\|\by^k-\by^{k-1}\|}.
\]
Thus, for all $k$ sufficiently large,
\[
{\|\by^{k+1}-\by^{k}\|}^2\leq \Psi\Theta_{k,k+1}\|\by^k-\by^{k-1}\|.
\]
This, together with the fact that $2\sqrt{ab}\le a+b$ for all $a,b\ge 0$, yields, for all $k$ sufficiently large,
\[
2\|\by^{k+1}-\by^{k}\| \le \|\by^k-\by^{k-1}\|+\Psi\Theta_{k,k+1}\le \|\by^k-\by^{k-1}\|+\Psi\xi(F(\la_*,\by^{\ell+1})-F(\la_*,\by^*)).
\]
Therefore, using the definition of $\Theta_{k,k+1}$ and noting that $\chi_{\cs^{n-1}}\geq0$, we have  for all $\ell$ sufficiently large,
\begin{eqnarray*}
2\sum\limits_{j=\ell+1}^{k}\|\by^{j+1}-\by^{j}\|&\leq&\sum\limits_{j=\ell+1}^{k}\|\by^{j}-\by^{j-1}\|
+\Psi\sum\limits_{j=\ell+1}^{k}\Theta_{j,j+1}\\
&\leq&\sum\limits_{j=\ell+1}^{k}\|\by^{j+1}-\by^{j}\|+\|\by^{l+1}-\by^{l}\| +\Psi\sum\limits_{j=\ell+1}^{k}\Theta_{j,j+1}\\
&=&\sum\limits_{j=\ell+1}^{k}\|\by^{j+1}-\by^{j}\|+\|\by^{l+1}-\by^{l}\| +\Psi\Theta_{\ell+1,k+1}
\end{eqnarray*}
\begin{eqnarray*}
&\le&\sum\limits_{j=\ell+1}^{k}\|\by^{j+1}-\by^{j}\|+\|\by^{l+1}-\by^{l}\| +\Psi\xi\big(F(\la_*,\by^{\ell+1})-F(\la_*,\by^*)\big).
\end{eqnarray*}
This implies that, for all $\ell$ sufficiently large,
\[
\sum\limits_{j=\ell+1}^{k}\|\by^{j+1}-\by^{j}\|\leq\|\by^{l+1}-\by^{l}\| +\Psi\xi\big(F(\la_*,\by^{\ell+1})-F(\la_*,\by^*)\big).
\]
Taking $k\to\infty$ we have
\[
\sum\limits_{k=1}^{\infty}\|\by^{k+1}-\by^{k}\|\le\infty.
\]
This shows that the sequence $\{\by^k\}$ is cauchy sequence and thus the sequence $\{\by^k\}$ is convergent. Therefore,  the sequence $\{\by^k\}$ converge to $\by^*$, which is a critical point of $F(\la_*,\cdot)$.
\end{proof}

\section{Extensions}\label{sec4}
In this section, we extend the geometric proximal gradient method proposed in Section \ref{sec2} to some sparse least squares regression with rectangular stochastic matrix constraint and the inverse eigenvalue problem for stochastic matrices.

\subsection{Rectangular stochastic matrix constrained least squares regression} \label{sec41}
A matrix $X=(x_{ij})\in\Rnr$ is called  a rectangular column (row) stochastic matrix if all its entries are nonnegative and all its column (row) sum equals one, i.e., ${\bf 1}_n^TX={\bf 1}_r^T$ (or $X{\bf 1}_r= {\bf 1}_n$).

In the following, we consider the following singly rectangular stochastic matrix constrained least squares regression:
\BE\label{lsr-sm}
\begin{array}{lc}
\min\limits_{X\in\Rnr}  &  \displaystyle \frac{1}{2}\|AX-B\|_F^2 \\[2mm]
\mbox{s.t.} &{\bf 1}_n^TX={\bf 1}_r^T, \quad X\ge 0,
 \end{array}
\EE
where $A\in\Rmn$ and $X\ge 0$ means that $X$ is a entry-wise nonnegative matrix. Such problem arises in sparse hyperspectral unmixing \cite{LB08}.

In \cite{IB12}, Iordache et al. gave the ADMM method for solving the following  total variation  regularization problem:
\[
\begin{array}{lc}
\min\limits_{X\in\Rnr}  &  \displaystyle \frac{1}{2}\|AX-B\|_F^2 +\la \|X\|_1+\la_{TV}TV(X) \\[2mm]
\mbox{s.t.} &X\ge 0,
 \end{array}
\]
where $\la,\la_{TV}>0$ are two regularization parameters and $TV(X):=\sum_{i,j}\|\bx_i-\bx_j\|_1$ with $\bx_j$ being the $j$-th column of $X$. However, the constraint ${\bf 1}_n^TX={\bf 1}_r^T$ is not involved.
In \cite{MI12}, Moussaoui et al. presented a primal-dual interior point method for solving the following regularized model :
\[
\begin{array}{lc}
\min\limits_{X\in\Rnr}  &  \displaystyle \frac{1}{2}\|AX-B\|_F^2 +\la R(X) \\[2mm]
\mbox{s.t.} &X\ge 0,
 \end{array}
\]
where $\la>0$ is the regularization parameter and $R(X)$ is used to estimate the abundance maps in hyperspectral imaging.

To find a sparse solution to problem  \eqref{lsr-sm}, we reformulate problem  \eqref{lsr-sm} as an nonconvex and nonsmooth minimization problem over a Riemannian manifold. We first note that
\[
\{X\in\Rnr\; |\; \mbox{${\bf 1}_n^TX={\bf 1}_r^T$, $X\ge 0$}\}=
\{Y\odot Y \in\Rnr\; |\; \mbox{$Y\in\ob(n,r)$}\},
\]
where the set $\ob(n,r)$ is the rectangular oblique manifold:
\[
\ob(n,r):=\{Y \in\Rnr\; |\; \mbox{$\diag(Y^TY)=I_r$}\}.
\]
Instead of  problem  \eqref{lsr-sm}, one may consider the following $\ell_1$ regularization problem:
\BE\label{lsr-L1R}
\begin{array}{lc}
\min\limits_{Y\in\Rnr}  &  \displaystyle \frac{1}{2}\|A(Y\odot Y)-B\|_F^2+\la\|Y\|_1\\[2mm]
\mbox{s.t.} & Y\in\ob.
\end{array}
\EE

Let
\BE\label{def:pq-lsr}
p(Y):=\frac{1}{2}\|A(Y\odot Y)-B\|_F^2,\quad
q(\lambda,Y):=\lambda \|Y\|_1,\quad
G(\lambda,Y):=p(Y)+q(\lambda,Y)+\Phi(Y),
\EE
where $\Phi$ is a characteristic function of  $\ob(n,r)$ defined by
\[
\Phi(Y)=\left\{
\begin{array}{cl}
0,         &   Y\in\ob(n,r),\\
+\infty,  & \mbox{otherwise}.\\
\end{array} \right.
\]

Then one may apply Algorithm \ref{pgm} to problem \eqref{lsr-L1R}, where in each iteration, one needs to find the explicit expression of $\overline{Y}^{k}\in\ob(n,r)$ such that
\BE\label{lsr-ques3}
\overline{Y}^{k}=\argmin_{Y\in\ob(n,r)} \Big\{p(Y^k)+\big\langle Y-Y^k , \nabla p(Y^k) \big\rangle_F +\frac{1}{2\alpha}\|Y-Y^k\|_F^2+{\lambda_k}\|Y\|_1\Big\}.
\EE
For any integer $k\ge 0$, let
\[
Y^k:=[\by_1^k,\by_2^k,\ldots,\by_r^k]
\quad\mbox{and}\quad
\nabla p(Y^k):=[\nabla p_1(Y^k),\nabla p_2(Y^k),\ldots, \nabla p_r(Y^k)].
\]
Thus $Y^k\in\ob(n,r)$ if and only if $\by_j^k\in\cs^{n-1}$ for $j=1,\ldots,n$. Hence, we can solve \eqref{lsr-ques3} by solving
\BE\label{ques32}
\overline{\by}_j^{k}=\argmin_{\by_j\in\cs^{n-1}} \Big\{\big\langle \by_j-\by_j^k , \nabla p_1(Y^k) \big\rangle +\frac{1}{2\alpha}\|\by_j-\by_j^k\|^2+{\lambda_k}\|\by_j\|_1\Big\},
\EE
for $j=1,\ldots, r$, which have explicit expressions as $\overline{\by}^{k}$ defined in \eqref{ques1} of Algorithm \ref{pgm}.

In addition, it is obvious that for any $Y\in\ob(n,r)$,
\BE\label{p:grad}
\nabla p(Y)=2\big(A^TA(Y\odot Y)-A^TB\big)\odot Y.
\EE
As in section \ref{sec32}, we can establish the global Lipschitz continuity of $\nabla p$ as follows.
\begin{lemma}\label{p:lc-lsr}
Let $\cb_2(0):=\{Y=[\by_1,\by_2,\ldots,\by_r]\in\Rnr\; |\; \mbox{$\|\by_j\|\le 1$ for $j=1,\ldots,n$}\}$ be the  closed subset in $\Rnn$. Then, for the function $p$ defined in  \eqref{def:pq-lsr}, there exists a constant $L_p>0$ such that
\[
\|\nabla p(Y_1)-\nabla p(Y_2)\|_F\le L_p\|Y_1-Y_2\|_F,\quad  \forall Y_1,Y_2\in\cb_2(0),
\]
where $L_p=6n\|A^TA\|_2+2\|A^TB\|_F$.
\end{lemma}
Finally, one may develop the global convergence of Algorithm \ref{pgm} for problem \eqref{lsr-L1R} as in section \ref{sec33}.

\subsection{Inverse eigenvalue problem for stochastic matrices}
In this subsection, we consider  the inverse eigenvalue problem for stochastic matrices of the reconstruction of a sparse row stochastic matrix from the prescribed stationary distribution vectors. Such problem arises in the inverse problem of reconstructing  a Markov Chain from the prescribed steady-state probability  distribution \cite{CC09} and the construction of probabilistic Boolean networks  \cite{CJ12}. It  was also mentioned in \cite[p. 104]{CG05} as a stochastic inverse eigenvalue problem.

The inverse eigenvalue problem for stochastic matrices with the prescribed stationary distribution vectors $\{\bfd_i\in\Rn\}_{i=1}^m$ aims to reconstruct a  matrix $X\in\Rnn$ such that
\[
\bfd_i^TX=\bfd_i^T,\quad X{\bf 1}_n={\bf 1}_n,\quad i=1,\ldots, m,\quad X\ge 0.
\]
Alternatively, we consider  the following least square regression problem:
\BE\label{siep-min}
\begin{array}{lc}
\min\limits_{X\in\Rnn}  &  \displaystyle \frac{1}{2}\|DX-D\|_F^2  \\[2mm]
\mbox{s.t.} &X{\bf 1}_n={\bf 1}_n,\quad X\ge 0,
\end{array}
\EE
where $D:=[\bfd_1,\ldots,\bfd_m]^T$.

To find a sparse solution to problem \eqref{siep-min},  as in section \ref{sec41}, we consider the following $\ell_1$ regularization problem:
\BE\label{siep-L1R}
\begin{array}{lc}
\min\limits_{Y\in\Rnn}  &  \displaystyle \frac{1}{2}\|D(Y\odot Y)-D\|_F^2+\la\|Y\|_1\\[2mm]
\mbox{s.t.} & Y\in\ob.
\end{array}
\EE
where the set $\ob$ is the oblique manifold \cite{AMS08}:
\[
\ob:=\{Y \in\Rnn\; |\; \mbox{$\diag(YY^T)=I_n$}\}.
\]

Let
\BE\label{def:pq}
p(Y):=\frac{1}{2}\|D(Y\odot Y)-D\|_F^2,\quad
q(\lambda,Y):=\lambda \|Y\|_1,\quad
G(\lambda,Y):=p(Y)+q(\lambda,Y)+\Phi(Y),
\EE
where $\Phi$ is a characteristic function of  $\ob$ defined by
\[
\Phi(Y)=\left\{
\begin{array}{cl}
0,         &   Y\in\ob,\\
+\infty,  & \mbox{otherwise}.\\
\end{array} \right.
\]
We can use Algorithm \ref{pgm}  to  problem \eqref{siep-L1R}, where we need to find the explicit expression of $\overline{Y}^{k}\in\ob$ defined as in \eqref{lsr-ques3}. For any integer $k\ge 0$, let
\[
Y^k:=[\by_1^k,\by_2^k,\ldots,\by_n^k]^T
\quad\mbox{and}\quad
\nabla p(Y^k):=[\nabla p_1(Y^k),\nabla p_2(Y^k),\ldots, \nabla p_n(Y^k)]^T.
\]
Then
$Y^k\in\ob$ if and only if $\by_j^k\in\cs^{n-1}$ for $j=1,\ldots,n$, where  $\by_j^k\in\cs^{n-1}$ is determined by \eqref{ques32}, which has an explicit expression.

As in section \ref{sec32}, we can establish the  global Lipschitz continuity of $\nabla p$, where the function $p$ is  defined in \eqref{def:pq}. We have for any $Y\in\ob$,
\BE\label{p:grad}
\nabla p(Y)=2\big(D^TD(Y\odot Y)-D^TD\big)\odot Y.
\EE

By using the similar proof of  Lemma \ref{p:lc-lsr}, we have the following result on the global Lipschitz continuity of $\nabla p$.
\begin{lemma}\label{p:lc}
Let $\cb_3(0):=\{Y=[\by_1,\by_2,\ldots,\by_n]^T\in\Rnn\; |\; \mbox{$\|\by_j\|\le 1$ for $j=1,\ldots,n$}\}$ be the  closed subset in $\Rnn$. Then, for the function $p$ defined in  \eqref{def:pq}, there exists a constant $L_p>0$ such that
\[
\|\nabla p(Y_1)-\nabla p(Y_2)\|_F\le L_p\|Y_1-Y_2\|_F,\quad  \forall Y_1,Y_2\in\cb_3(0),
\]
where $L_p=6n\|D^TD\|_2+2\|D^TD\|_F$.
\end{lemma}

Therefore, one may solve problem \eqref{siep-L1R} via Algorithm \ref{pgm}, whose global convergence can be established as in section \ref{sec33}.
\section{Numerical experiments}\label{sec5}
In this section, we report the numerical performance of Algorithm \ref{pgm}  for solving the sparse least squares regression problem \eqref{pbn-r}. To illustrate the effectiveness of our method, we compare the proposed algorithm with the  projection-based gradient descent method (PG) \cite{WW15} for solving problem \eqref{intro} and the ADMM method  \cite{BF10}  for solving problem \eqref{pro:mp}. Our numerical tests are implemented by running {\tt MATLAB R2019a} on a personal laptop (Intel Core i7-8559U @ 2.7GHz, 16 GB RAM).

In our numerical tests, for Algorithm \ref{pgm}, the  PG method, and the ADMM method, the initial guess is set to be  $\bx^0={\bf 1}_n/n$ and $\by^0={\bf 1}_n/\sqrt{n}$ ($X^0=(x_{ij}^0)$ with  $x_{1j}^0=\cdots= x_{nj}^0=1/\sqrt{n}$ for $j=1,\ldots,r$ and $Y^0=(y_{ij}^0)$ with $y_{1j}^0=\cdots= y_{nj}^0=1/\sqrt{n}$ for $j=1,\ldots,r$ for problem \eqref{lsr-sm}, or $X^0=(x_{ij}^0)$ with $x_{ij}^0=1/n$ for $i,j=1,\ldots,n$ and $Y^0=(y_{ij}^0)$ with $y_{ij}^0=1/\sqrt{n}$ for  $i,j=1,\ldots,n$ for problem \eqref{siep-min}, the stopping criterion is set to be
\[
\frac{\|\by^k\odot\by^k-\by^{k-1}\odot\by^{k-1}\|}{\|\by^{k-1}\odot\by^{k-1}\|}\le{\tt tol}
\quad\mbox{and}\quad
\frac{\|{\bf x}^k-{\bf x}^{k-1}\|}{\|{\bf x}^{k-1}\|}\le{\tt tol},
\]
or
\[
\frac{\|Y^k\odot Y^k-Y^{k-1}\odot Y^{k-1}\|_F}{\|Y^{k-1}\odot Y^{k-1}\|_F}\le{\tt tol}
\quad\mbox{and}\quad
\frac{\|X^k-X^{k-1}\|_F}{\|X^{k-1}\|_F}\le{\tt tol},
\]
and the largest number of iterations is set to be {\tt ITmax}, where ``{\tt tol}" is a prescribed tolerance.
For Algorithm \ref{pgm}, we also set  $\rho_1=\rho_3=0.9$, $\rho_2=0.6$, $\gamma_2=10^{-5}$, $\gamma_1=0.9/(L_f+\gamma_2)>0$,  $\delta_1=4.0$, and $\delta_2=10^{-4}$.
Let `{\tt nnz.}', `{\tt ct.}',  `{\tt kkt.}', and `{\tt obj.}',  denote the number of nonzeros in the computed solution, the total computing time in seconds, the KKT residual of respective models, and the objective function values at the final iterates of the corresponding algorithms, accordingly.

We first consider the following  Lasso problem  as in \cite{LL19}.
\begin{example}\label{ex51}
Let $\bb\in\Rm$ be generated by the linear regression model:
\[
\bb=A\bx^*+\nu\bn, \quad \bn\sim N({\bf 0}, I),
\]
where the rows of $A\in \Rmn$ are generated by a  Gaussian distribution $N({\bf 0}, I_n)$ and
the regression coefficient vector $\bx^*=|\bar{\bx}|/\|\bar{\bx}\|_1$ with $\bar{\bx}\in\Rn$ being a sparse normally distributed random vector  generated by the {\tt MATLAB} built-in function {\tt sprandn} with $5\%$  uniformly distributed nonzero entries. We set $\nu=0.001\|A\bx\|/\|\bn\|$. We report our numerical results for $m=20j$ and $n=300j$ with $j=1, 2, 3, 4, 5$.
\end{example}

In Table \ref{table51}, we report the numerical results for  Example \ref{ex51} with ${\tt tol}=10^{-4}$ and ${\tt ITmax}=2000$. We observe from  Table \ref{table51} that both the PG method and the ADMM method need less running time than  Algorithm \ref{pgm}, where the PG method is the most efficient method in terms of the running time and the ADMM method is the most effective method in terms of the  objective function value. However, Algorithm \ref{pgm} can find a much  sparser solution than the PG method and the ADMM method with acceptable running time, where  Algorithm \ref{pgm} with fixed $\la$ can find the most sparse solution with a relatively  large objective function value while Algorithm \ref{pgm} can reach a good  tradeoff between  sparsity and objective function value.

\begin{table}[!ht]\renewcommand{\arraystretch}{1.0} \addtolength{\tabcolsep}{1.0pt}
\begin{center}{\small
  \begin{tabular}[c]{|c|c|c|c|c|}     \hline
\multicolumn{5}{|c|}{$j=1$ and $\|\bx^*\|_0=15$}\\ \hline
Alg. & {\tt nnz. }  & {\tt kkt.}  & {\tt obj.}  & {\tt ct.}   \\ \hline
PG & $75$           &$4.3348\times 10^{-3}$  &  $9.3194\times 10^{-7} $                                  &$0.0053$            \\ \hline
ADMM &                          $299$           &$9.0578\times 10^{-4}$  &                                         $1.3590\times 10^{-9} $                                  &$0.0046$            \\ \hline
Alg. \ref{pgm} with fixed $\lambda=10^{-2}$ &     $15$            &$5.44691\times 10^{-2}$  &                                         $9.7324\times 10^{-4} $                                  &$1.1132$            \\ \hline
Alg. \ref{pgm} with $\lambda^0=10^{-2}$ & $33$                 &$1.3475\times 10^{-3}$  &                                         $7.7540\times 10^{-7} $                                  &$0.6467$            \\ \hline
\multicolumn{5}{|c|}{$j=2$ and $\|\bx^*\|_0=30$}\\ \hline
Alg. & {\tt nnz. }  & {\tt kkt.}  & {\tt obj.}  & {\tt ct.}   \\ \hline
PG &                              $225$           &$5.2341\times 10^{-3}$  &                                         $2.7488\times 10^{-7} $                                  &$0.0059$            \\ \hline
ADMM &                                     $600$              &$2.0430\times 10^{-6}$  &                                         $3.0117\times 10^{-15} $                                  &$0.0207$            \\ \hline
Alg. \ref{pgm} with fixed $\lambda=0.7071\times 10^{-2}$ &        $33$             &$2.3521\times 10^{-2}$  &                                         $1.7426\times 10^{-3} $                                  &$1.5774$            \\ \hline
Alg. \ref{pgm} with $\lambda^0=0.7071\times 10^{-2}$ &   $52$                 &$4.9274\times 10^{-4}$  &                                         $2.2965\times 10^{-7} $                                  &$0.7816$            \\ \hline

\multicolumn{5}{|c|}{$j=3$ and $\|\bx^*\|_0=45$}\\ \hline
Alg. & {\tt nnz. }  & {\tt kkt.}  & {\tt obj.}  & {\tt ct.}   \\ \hline
PG &                        $325$           &$6.2055\times 10^{-3}$  &                                         $2.1367\times 10^{-7} $                                  &$0.0084$            \\ \hline
ADMM &                             $900$                          &$1.9741\times 10^{-6}$  &                                         $2.4255\times 10^{-15} $                                  &$0.0270$            \\ \hline
Alg. \ref{pgm} with fixed $\lambda=0.5774\times 10^{-2}$ &  $47$                &$3.9036\times 10^{-3}$  &                                         $9.5213\times 10^{-4} $                                  &$1.7178$            \\ \hline
Alg. \ref{pgm} with $\lambda^0=0.5774\times 10^{-2}$ & $90$                 &$4.8763\times 10^{-4}$  &                                         $3.8463\times 10^{-8} $                                  &$0.8099$            \\ \hline

\multicolumn{5}{|c|}{$j=4$ and $\|\bx^*\|_0=57$}\\ \hline
Alg. & {\tt nnz. }  & {\tt kkt.}  & {\tt obj.}  & {\tt ct.}   \\ \hline
PG &                          $578$           &$5.7746\times 10^{-3}$  &                                         $8.4740\times 10^{-8} $                                  &$0.0178$            \\ \hline
ADMM & $1200$                                       &$3.3031\times 10^{-7}$  &                                         $4.9351\times 10^{-17} $                                  &$0.0382$            \\ \hline
Alg. \ref{pgm} with fixed $\lambda=0.5\times 10^{-2}$ &    $61$             &$2.2108\times 10^{-3}$  &                                         $1.2072\times 10^{-3} $                                  &$3.1507$            \\ \hline
Alg. \ref{pgm} with $\lambda^0=0.5\times 10^{-2}$ & $123$                 &$4.0730\times 10^{-4}$  &                                         $5.2864\times 10^{-8} $                                  &$0.8082$            \\ \hline

\multicolumn{5}{|c|}{$j=5$ and $\|\bx^*\|_0=72$}\\ \hline
Alg. & {\tt nnz. }  & {\tt kkt.}  & {\tt obj.}  & {\tt ct.}   \\ \hline
PG &                     $494$           &$1.0113\times 10^{-2}$  &                                         $3.5405\times 10^{-7} $                                  &$0.0138$            \\ \hline
ADMM & $1500$                                       &$1.7913\times 10^{-6}$  &                                         $1.1327\times 10^{-15} $                                  &$0.0494$            \\ \hline
Alg. \ref{pgm} with fixed $\lambda=0.5\times 10^{-2}$ & $81$                                       &$2.7140\times 10^{-2}$  &                                         $1.0152\times 10^{-3} $                                  &$4.0649$            \\ \hline
Alg. \ref{pgm} with $\lambda^0=0.5\times 10^{-2}$ & $143$                 &$3.8337\times 10^{-4}$  &                                         $8.3343\times 10^{-8} $                                  &$2.1601$            \\ \hline
\end{tabular} }
\end{center}
\caption{Numerical results for Example \ref{ex51}.}\label{table51}
\end{table}

Next, we consider a numerical example in hyperspectral applications \cite{LL19}.
\begin{example}\label{ex52}
Suppose the simulated data $\bb\in\Rm$ be generated by
\[
\bb=A\bx^*+\bn,
\]
where $A$ is a $224\times 440$  Gaussian random matrix with zero-mean unit variance, the noise $\bn$ is generated  by zero-mean i.i.d. Gaussian sequences of random variables, and the true fractional abundance vector $\bx^*=|\bar{\bx}|/\|\bar{\bx}\|_1$ with $\bar{\bx}\in\Rn$ being a sparse normally distributed random vector  generated by {\tt sprandn} with $2\%$  uniformly distributed nonzero entries. The signal-to-noise ratio (SNR) is defined by
\[
SNR=10\log_{10}\frac{\E(\|A\bx^*\|^2)}{\E(\|\bn \|^2)},
\]
where the expectation $\E(\cdot)$ is approximated with sample mean over ten runs.
\end{example}

Table \ref{table52-2} display the numerical results for Example \ref{ex52} with different SNRs, $\lambda=10^{-2}$, $\la^0=3.0\times 10^{-2}$, ${\tt tol}=10^{-5}$ and ${\tt ITmax}=3000$, where RSNR means the reconstruction SNR, which is defined by
\[
RSNR=10\log_{10}\frac{\E(\|\bx^*\|^2)}{\E(\|\bx^*-\bx^\#\|^2)}.
\]
Here, $\bx^\#$ denotes the computed solution.

We see from Table \ref{table52-2} that the proposed algorithm provides a high-precision solution.

\begin{table}[!h]\renewcommand{\arraystretch}{1.0} \addtolength{\tabcolsep}{1.0pt}
\begin{center}{\small
  \begin{tabular}[c]{|c|l|c|l|c|l|c|l|c|}     \hline
& \multicolumn{2}{|c|}{PG}  & \multicolumn{2}{|c|}{ADMM} & \multicolumn{2}{|c|}{Alg. \ref{pgm} with fixed $\la$} & \multicolumn{2}{|c|}{Alg. \ref{pgm}} \\ \hline
{\tt SNR~(dB)} & {\tt  RSNR~(dB)} & {\tt ct.} & {\tt RSNR} & {\tt ct.} & {\tt RSNR} & {\tt ct.} &
 {\tt RSNR} & {\tt ct.}                     \\ \hline
40           &$47.9440$ & $0.0038$  &$46.9500$ & $0.0145$ & ${\bf 56.3501}$ & $0.9834$                                      &${\bf 54.8873}$ & $3.1782$ \\ \hline
50           &$58.3520$ & $0.0030$  &$58.1053$ & $0.0087$ &${\bf 66.4631}$ & $0.9139$                                      &${\bf 65.2922}$ & $3.9925$      \\ \hline
60           &$66.9302$ & $0.0025$    &$63.9141$ & $0.0047$  &${\bf 75.0466}$ & $1.0746$                                       &${\bf 73.3215}$ & $3.1640$  \\ \hline
\end{tabular} }
\end{center}
\caption{Numerical results for  Example \ref{ex52} with different SNRs (averaged over ten runs).} \label{table52-2}
\end{table}

We now focus on the following two numerical examples on the inverse problem of reconstructing a probabilistic Boolean network (PBN) from a prescribed transition probability matrix \cite{CC11,CJ12,DP19,WW15}.

\begin{example}\label{ex53}
We consider a three-gene network and the prescribed transition probability matrix is given by
$$ P_1=
\left(
\begin{array}{cccccccc}
    0.1200  &       0  &  0.6000  &  0.4200   &      0   &      0   &      0   &      0  \\
    0.2800  &       0  &       0  &  0.1800   &      0   &      0   &      0   &      0  \\
         0  &  0.4000  &       0  &       0   & 0.4000   & 0.1800   &      0   &      0  \\
         0  &       0  &       0  &       0   &      0   & 0.4200   &      0   & 0.6000  \\
    0.1800  &       0  &  0.4000  &  0.2800   &      0   &      0   &      0   &      0  \\
    0.4200  &       0  &       0  &  0.1200   &      0   &      0   &      0   &      0  \\
         0  &  0.6000  &       0  &       0   & 0.6000   & 0.1200   &      0   &      0  \\
         0  &       0  &       0  &       0   &      0   & 0.2800   & 1.0000   & 0.4000
\end{array}
\right).$$
In this PBN, there are $1024$ Boolean networks (BNs).
\end{example}

\begin{example} \label{ex54}
We consider a three-gene network and the prescribed transition probability matrix is given by
$$ P_2=
\left(
\begin{array}{cccccccc}
    0.5672  &  0.4328  &  0.2881   &      0  &  0.1447  &       0  &  0.4328  &       0 \\
         0  &       0  &  0.1447   &      0  &  0.2881  &       0  &       0  &       0 \\
         0  &       0  &       0   &      0  &       0  &       0  &       0  &  0.3776 \\
         0  &       0  &       0   & 0.4328  &       0  &       0  &       0  &  0.1896 \\
    0.4328  &  0.5672  &  0.3376   &      0  &  0.1896  &       0  &  0.5672  &       0 \\
         0  &       0  &  0.1896   &      0  &  0.3776  &       0  &       0  &       0 \\
         0  &       0  &       0   &      0  &       0  &  0.6657  &       0  &  0.2881 \\
         0  &       0  &       0   & 0.5672  &       0  &  0.3343  &       0  &  0.1447
\end{array}
\right).$$
In this PBN, there are $2048$ BNs.
\end{example}

In Figures \ref{fig53}--\ref{fig54} and Tables \ref{table53-1}--\ref{table54-1}, we report the numerical results for Examples \ref{ex53}--\ref{ex54} with $\lambda=\lambda^0=10^{-2}$, ${\tt tol}=10^{-5}$, and ${\tt ITmax}=3000$. From  Figures \ref{fig53}--\ref{fig54}, we observe that the proposed algorithm generates a much sparser solution than the other two methods. From Tables \ref{table53-1}--\ref{table54-1}, we also see that  Algorithm \ref{pgm} can achieve a good  tradeoff between  sparsity and objective function value.

\begin{figure}[!h]
 \centering
\includegraphics[width=0.45\textwidth]{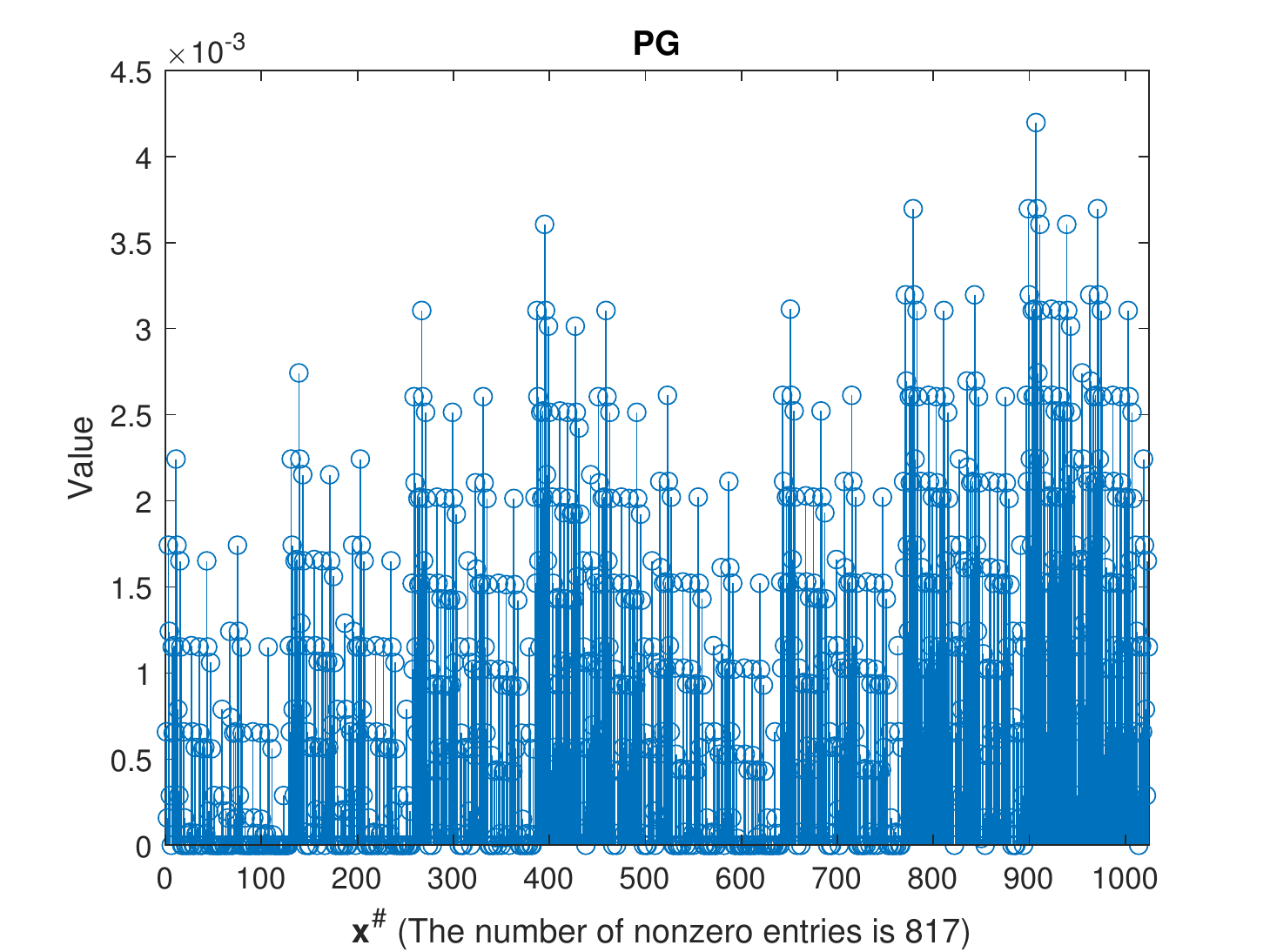}
\includegraphics[width=0.45\textwidth]{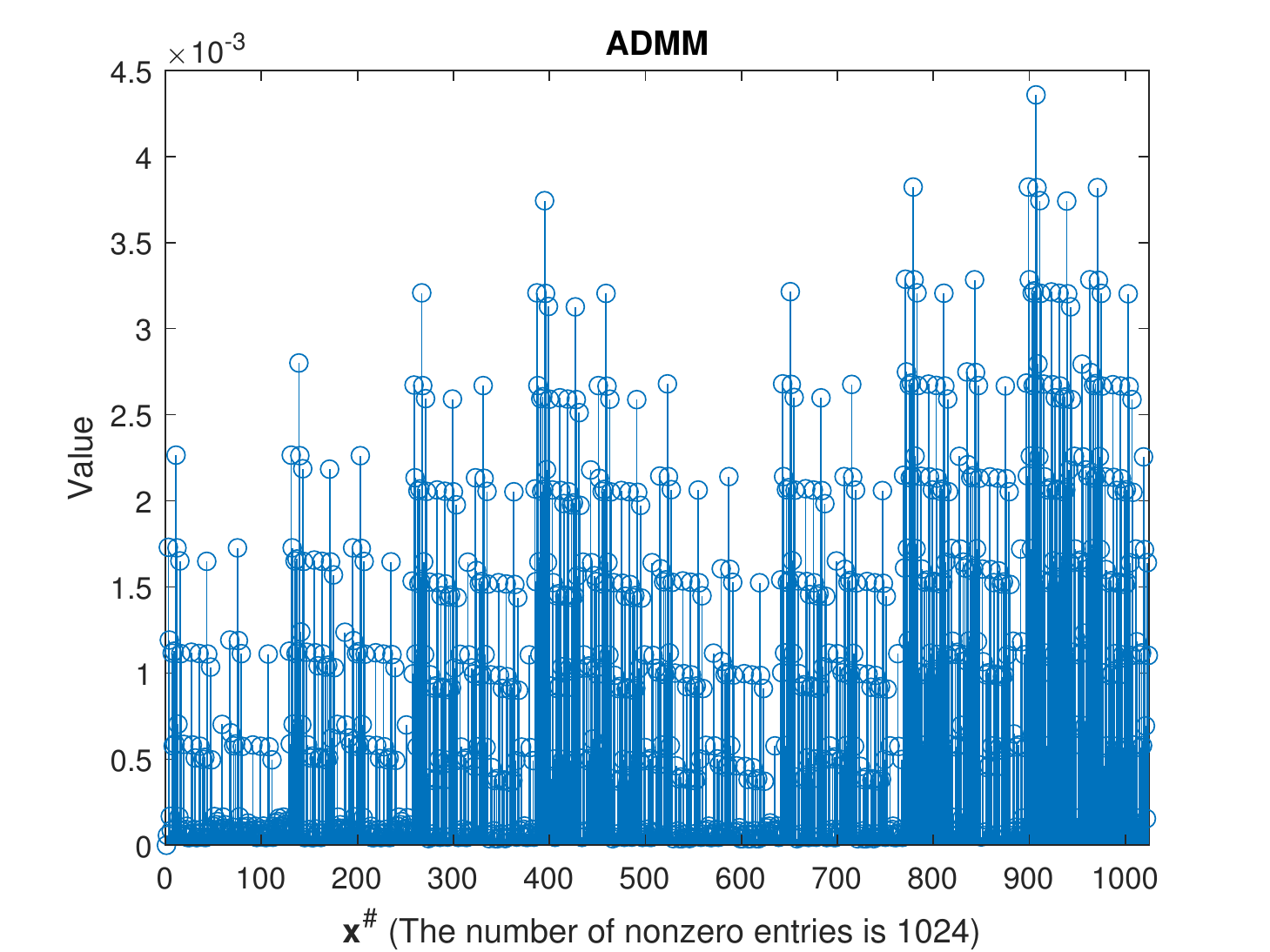}
\includegraphics[width=0.45\textwidth]{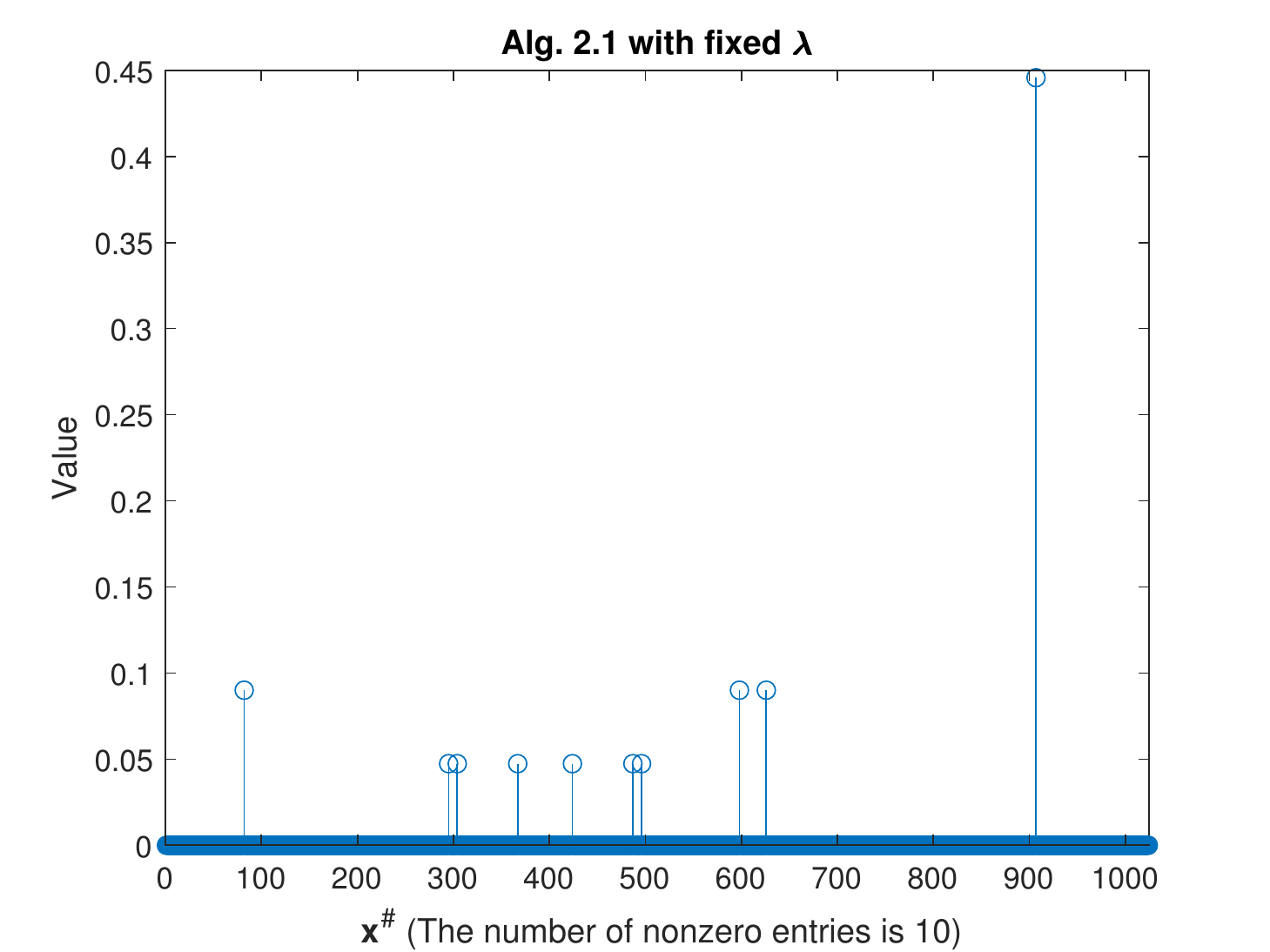}
\includegraphics[width=0.45\textwidth]{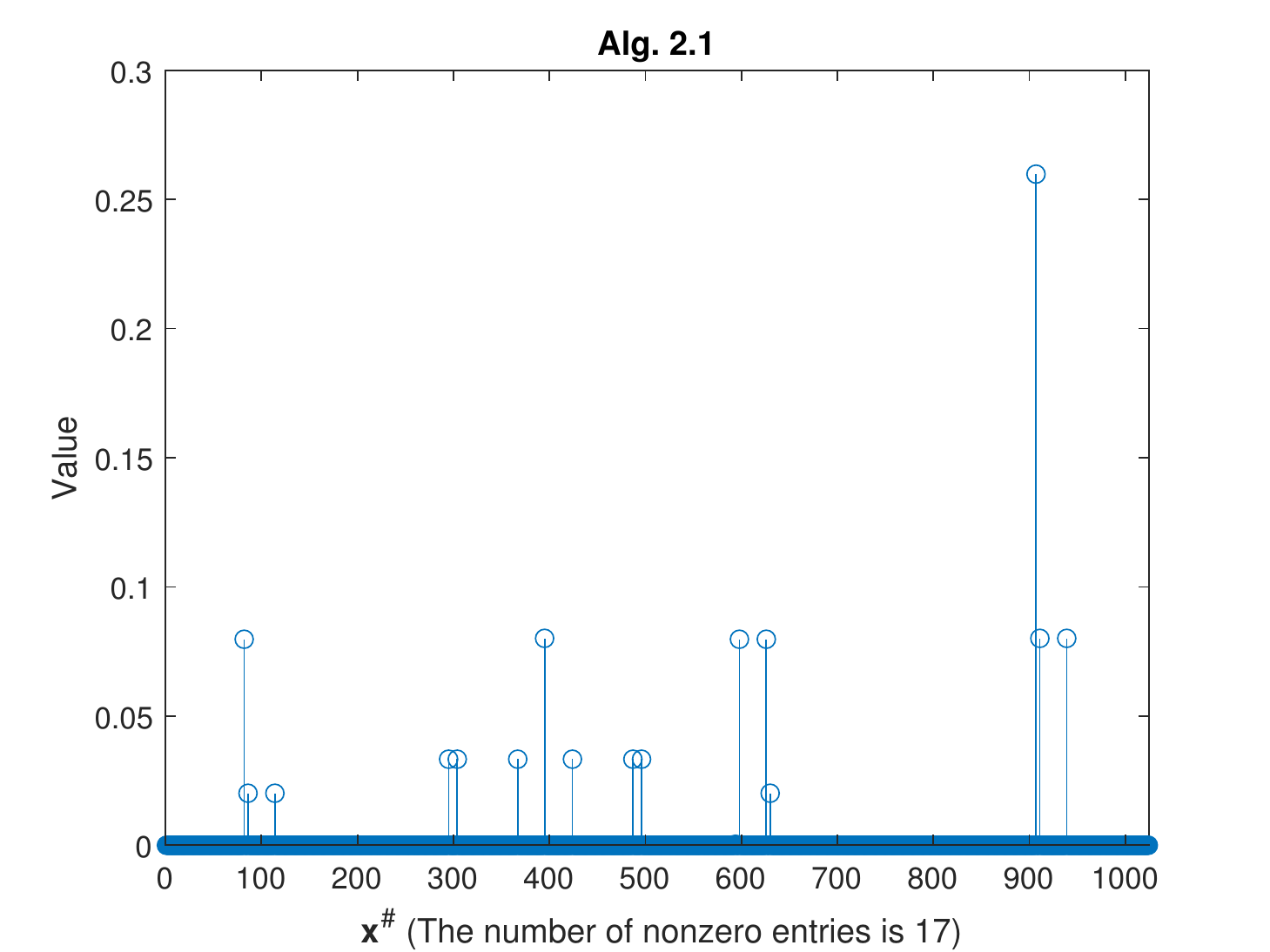}
 \caption{The probability distribution $\bx^\#$ for Example \ref{ex53}.}. \label{fig53}
 \end{figure}
 \begin{table}[!h]\renewcommand{\arraystretch}{1.0} \addtolength{\tabcolsep}{1.0pt}
\begin{center}{\small
  \begin{tabular}[c]{|c|c|c|c|c|}     \hline
& {\tt nnz.} & {\tt kkt.} & {\tt obj.}  & {\tt ct.}   \\ \hline
PG  &                          $817$           &$6.2553\times 10^{-5}$  &                                         $1.2039\times 10^{-11}$                                  &$0.0050$            \\ \hline
ADMM &                          $1024$           &$4.0931\times 10^{-8}$  &                                         $2.8067\times 10^{-18}$                                  &$0.0167$            \\ \hline
Alg. \ref{pgm} with fixed~$\lambda=10^{-2}$ &     $10$            &$4.2592\times 10^{-2}$  &                                         $2.9591\times 10^{-3} $                                  &$0.6303$            \\ \hline
Alg. \ref{pgm} with $\lambda^0=10^{-2}$ & $17$                 &$3.6060\times 10^{-5}$  &                                         $1.6691 \times 10^{-9}$                                  &$0.3647$            \\ \hline
\end{tabular} }
\end{center}
\caption{Numerical results for Example \ref{ex53}.}  \label{table53-1}
\end{table}
\begin{figure}[!h]
 \centering
\includegraphics[width=0.45\textwidth]{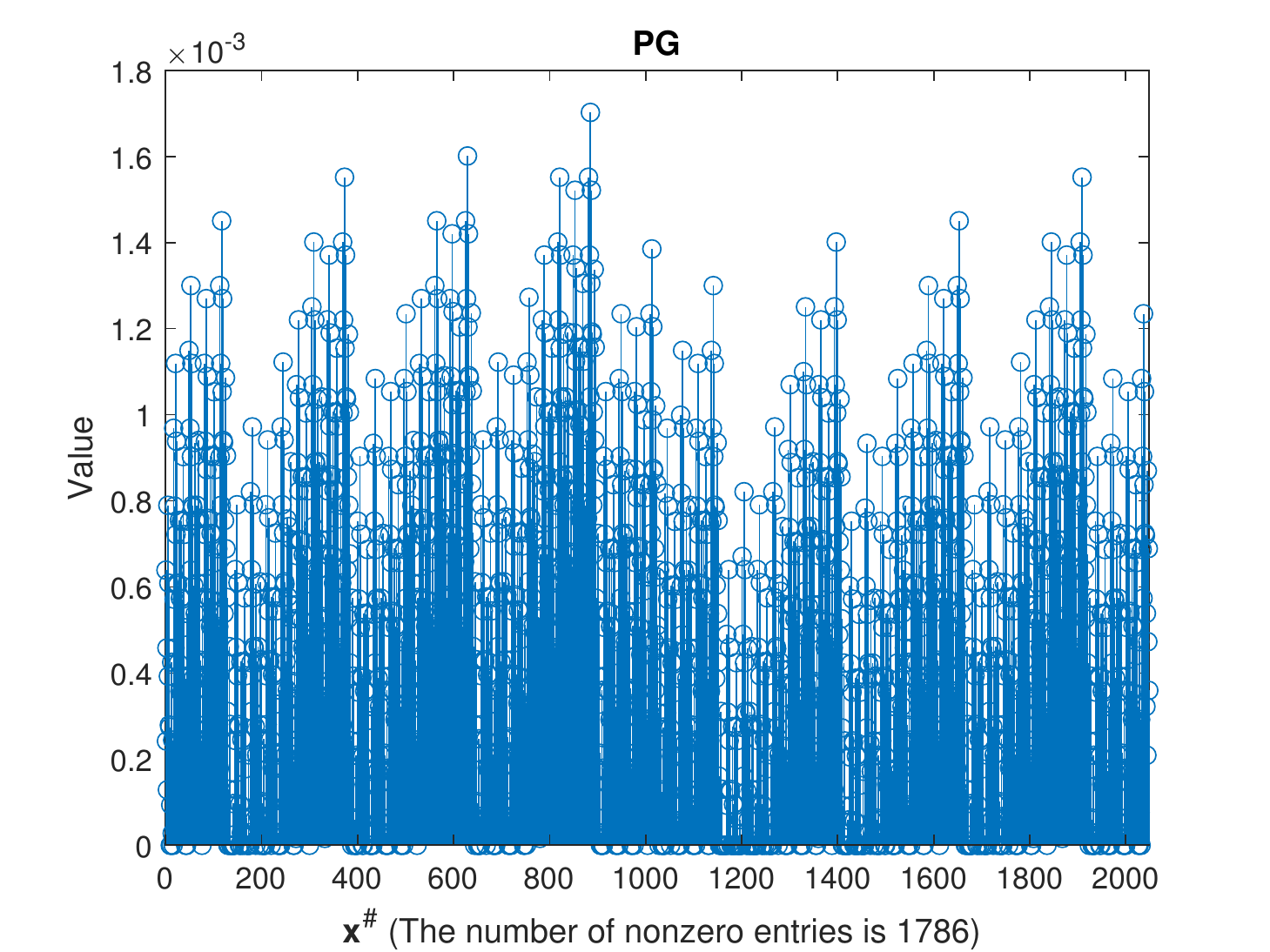}
\includegraphics[width=0.45\textwidth]{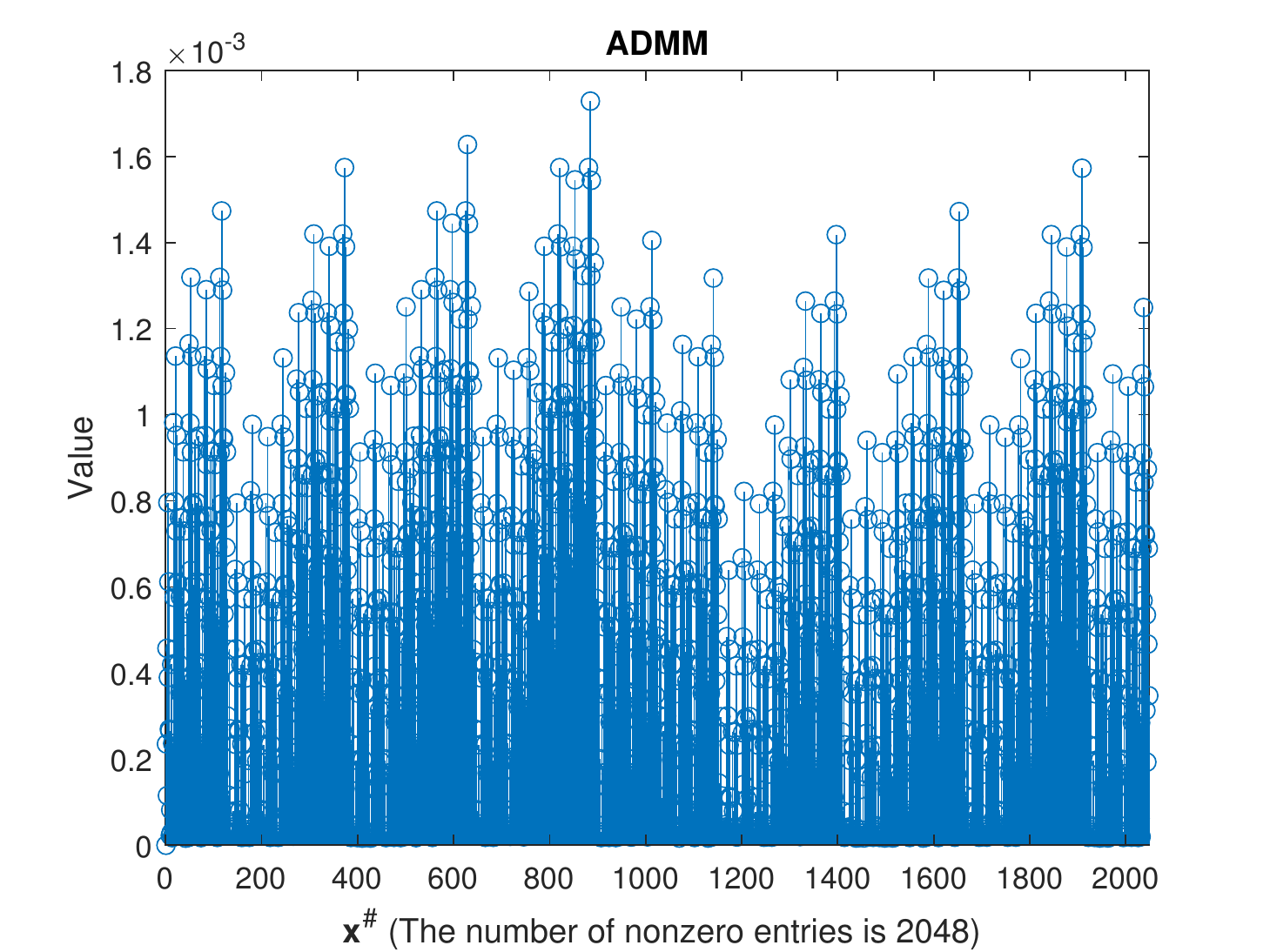}
\includegraphics[width=0.45\textwidth]{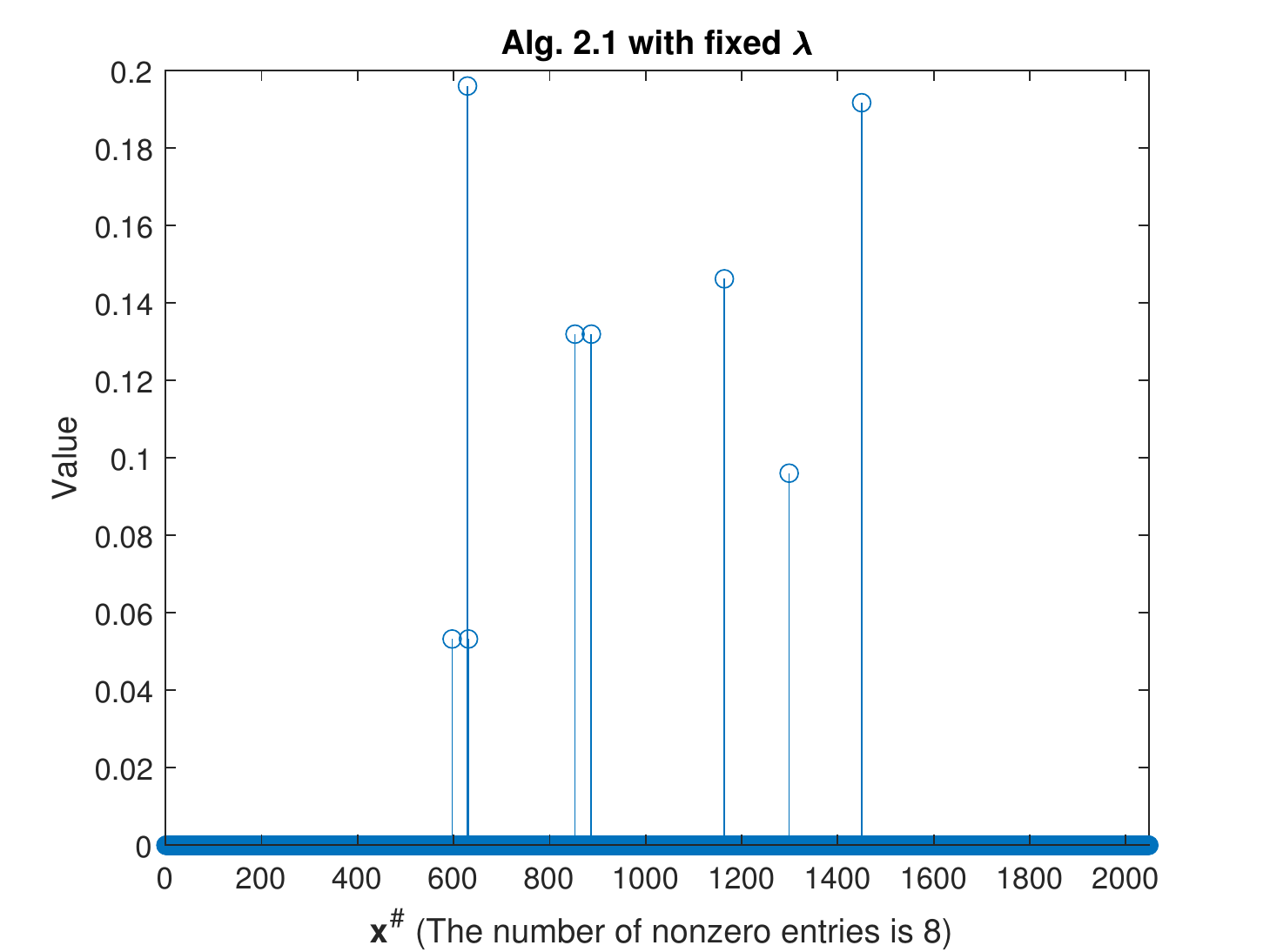}
\includegraphics[width=0.45\textwidth]{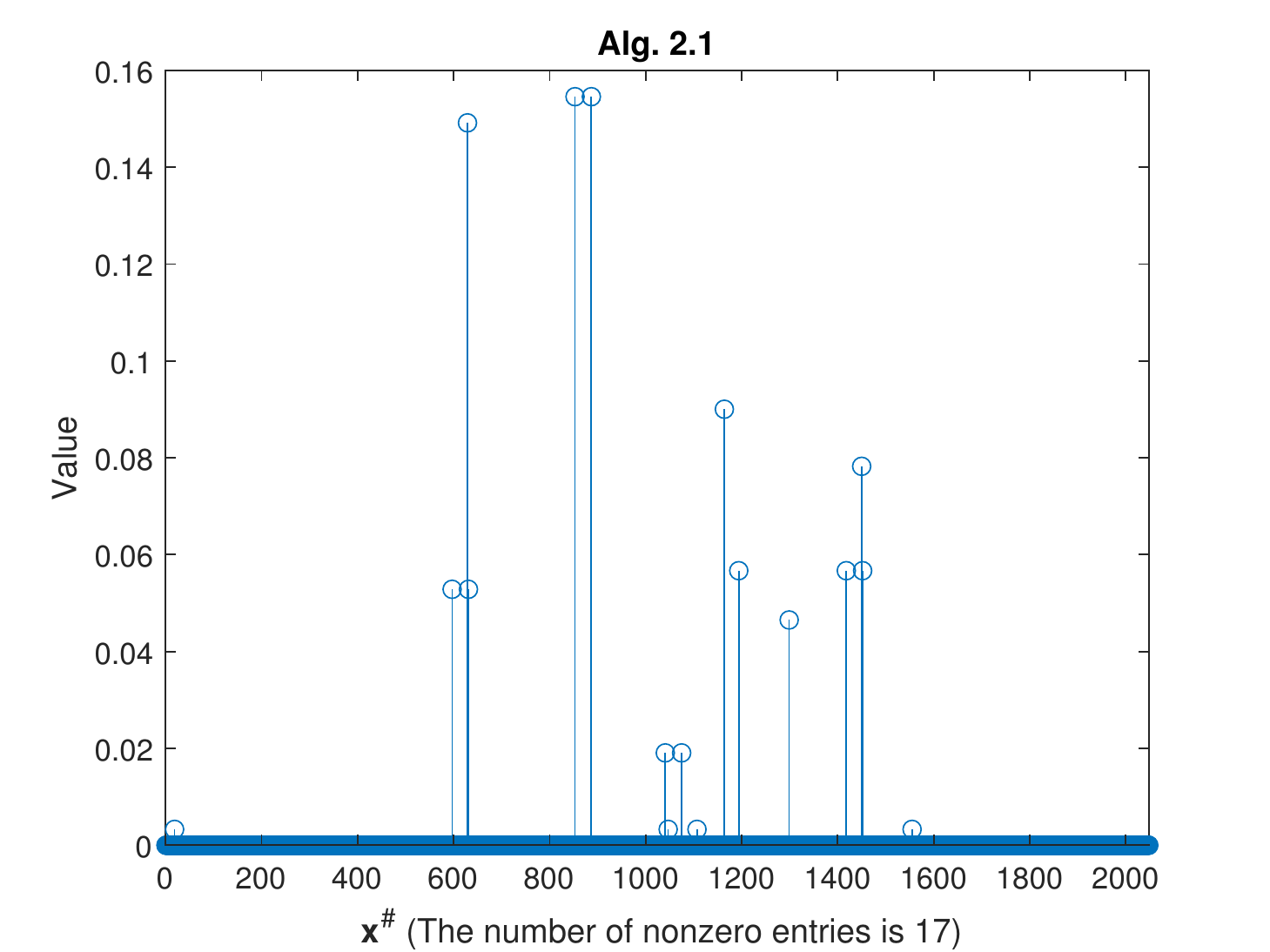}
 \caption{The probability distribution $\bx^\#$ for Example \ref{ex54}.}  \label{fig54}
 \end{figure}
\begin{table}[!h]\renewcommand{\arraystretch}{1.0} \addtolength{\tabcolsep}{1.0pt}
\begin{center}{\small
  \begin{tabular}[c]{|c|c|c|c|c|}     \hline
& {\tt nnz.} & {\tt kkt.} & {\tt obj.}  & {\tt ct.}   \\ \hline
PG  &                          $1786$           &$4.9093\times 10^{-4}$  &                                         $1.8462\times 10^{-4}$                                  &$0.0129$            \\ \hline
ADMM &                          $2048$           &$1.3327\times 10^{-8}$  &                                         $2.0000\times 10^{-4}$                                  &$0.0589$            \\ \hline
Alg. \ref{pgm} with fixed~$\lambda$ &     $8$            &$2.2379\times 10^{-2}$  &                                         $4.3463\times 10^{-4} $                                  &$1.0068$            \\ \hline
Alg. \ref{pgm} & $17$                 &$1.9996\times 10^{-2}$  &                                         $3.6000 \times 10^{-4}$                                  &$3.0860$            \\ \hline
\end{tabular} }
\end{center}
\caption{Numerical results for Example \ref{ex54}.} \label{table54-1}
\end{table}

Finally, we consider the following numerical example on the construction of a transition probability
matrix from a given stationary distribution \cite{CJ12}.
\begin{example}\label{ex56}
Construct a sparse transition probability matrix from  the given  stationary distribution vector:
\[
\bfd_1=(0.1282,0.2139,0.0667,0.1766,0.1758,0.0887,0.1324,0.0177)^T.
\]
\end{example}

We apply Algorithm \ref{pgm} to Example \ref{ex56} with  ${\tt tol}=10^{-6}$ and ${\tt ITmax}=6000$, where ${\tt rres.}=\|\bfd_1 X^\#-\bfd_1\|/\|\bfd_1\|$ denotes the relative residual  for the computed solution $X^\#$. Here, we set $\delta_2=10^{-5}$ and $\rho_3=0.95$ and the other parameters are set as above.
Then the computed  transition probability matrix via the PG method is given by
$$ X^\#=
\left(
\begin{array}{cccccccc}
    0.1293  &  0.2004  &  0.0782 &   0.1695   & 0.1688&    0.0965  &  0.1328 &   0.0246\\
    0.1261  &  0.2448  &  0.0409  &  0.1931  &  0.1920 &   0.0714 &   0.1319 &        0\\
    0.1272  &  0.1642  &  0.1007  &  0.1481  &  0.1478&    0.1102 &   0.1290 &   0.0727\\
    0.1290  &  0.2270  &  0.0587  &  0.1843  &  0.1834  &  0.0838 &   0.1338 &        0\\
    0.1291   & 0.2266  &  0.0590  &  0.1841   & 0.1832 &   0.0841  &  0.1338 &        0\\
    0.1280   & 0.1772  &  0.0926  &  0.1558  &  0.1553 &   0.1053   & 0.1304  &  0.0555\\
    0.1294  &  0.2029  &  0.0767  &  0.1709   & 0.1702  &  0.0956  &  0.1330  &  0.0213\\
    0.1256 &   0.1354   & 0.1185  &  0.1311   & 0.1310 &   0.1211   & 0.1261&    0.1111
             \end{array}
\right)$$
while the computed  transition probability matrix via  Algorithm \ref{pgm} with fixed $\la=10^{-3}$ is given by
$$  X^\#=
\left(
\begin{array}{cccccccc}
         0   &   1.0000  &       0     &    0     &    0   &      0   &      0     &    0\\
    0.3676 &        0   &      0    &     0     &    0     &    0   & 0.6324    &     0\\
         0   &      0      &    1.0000 &        0  &      0     &    0    &     0   &      0\\
    0.2842&         0   &      0     &    0   & 0.7158    &     0     &    0    &     0\\
         0   &      0      &   0   & 0.7005  &  0.2995    &     0    &     0    &     0\\
         0   &      0      &   0    &     0     &    0   & 1.0000   &      0   &      0\\
         0  &   0.5547   &      0  &  0.4453   &     0   &      0     &    0     &    0\\
         0 &   1.0000     &    0    &     0    &     0      &   0     &    0      &   0
 \end{array}
\right)$$
and the computed  transition probability matrices via  Algorithm \ref{pgm}  with  $\lambda_0=1.0\times 10^{-3}$ and $\lambda_0=5.0\times 10^{-4}$ are respectively given by
$$  X^\#=
\left(
\begin{array}{cccccccc}
          0  &  1.0000  &       0 &        0 &        0  &       0  &       0     &    0  \\
     0.3644  &       0  &       0 &        0 &        0  &       0  &  0.6356     &    0 \\
          0  &       0  &  1.0000 &        0 &        0  &       0  &       0     &    0 \\
     0.3046  &       0  &       0 &   0.2240 &   0.4714  &       0  &       0     &    0 \\
          0  &       0  &       0 &   0.4534 &   0.5466  &       0  &       0     &    0 \\
          0  &       0  &       0 &        0 &        0  &  1.0000  &       0     &    0 \\
          0  &  0.5403  &       0 &   0.4597 &        0  &       0  &       0     &    0 \\
          0  &  1.0000  &       0 &        0 &        0  &       0  &       0     &    0
\end{array}
\right),$$
$$ X^\#=
\left(
\begin{array}{cccccccc}
     0            &0.3081     &    0       &  0           & 0               &0.6919   &      0        &  0   \\
     0.3811   &      0        & 0           &0            & 0                 & 0          &0.6189     &    0\\
   0               & 0.5008    &0.4992   &      0     &    0               &0          &0               & 0  \\
   0.2644      &   0           & 0           &0.3018  &  0.4339        & 0         &0                &0\\
     0              & 0            &  0           &0.4358  &  0.5642         &0        & 0                & 0\\
   0                 &0.6233    &0.3767    &     0       &  0                 & 0      &   0              &0\\
    0               &0.6474      &   0           &0.3526    &     0            & 0       &  0              & 0\\
    0                &0                &0            &  0         &0                    &  0        & 0    &1.0000
 \end{array}
\right).$$
\begin{table}[!h]\renewcommand{\arraystretch}{1.0} \addtolength{\tabcolsep}{1.0pt}
\begin{center}{\small
  \begin{tabular}[c]{|c|c|c|c|}     \hline
Alg. & {\tt rres.} & {\tt nnz.}  &{\tt ct.}  \\ \hline
PG & $2.0325\times 10^{-5}$ & $61$  & 0.0188      \\ \hline
Alg. \ref{pgm} with fixed $\lambda$ &$5.0415\times 10^{-2} $& $12$ &0.0984\\ \hline
Alg. \ref{pgm} with $\lambda_0=1.0\times 10^{-3}$ & $4.9338\times 10^{-2}$ & $13$ & 0.0896  \\   \hline
Alg. \ref{pgm} with $\lambda_0=5.0\times 10^{-4}$ & $3.7022\times 10^{-5}$ & $16$ & 0.2291  \\   \hline
\end{tabular}}
\end{center}
\caption{Numerical results for Example \ref{ex56}.}  \label{table56}
\end{table}

The numerical results for  Example  \ref{ex56} are listed in Table \ref{table56}. From Table \ref{table56}, we can observe  that the computed solution via Algorithm \ref{pgm} with fixed/varied regularized parameter is much sparser than the PG method. We  point out that, for Algorithm \ref{pgm} with varied regularized parameter, a good tradoff between sparsity and residual can be obtained if an initial guess of the regularized parameter is selected appropriately.

\section{Concluding remarks}\label{sec6}
In this paper,  we have considered the sparse least squares regression problem with probabilistic simplex constraint, which is reformulated as a $\ell_1$ regularized minimization problem over the unit sphere. Then a geometric proximal gradient method is proposed for solving the  regularized problem. The global convergence of the proposed method is established under some mild assumptions. In each iteration of our method, we have derived the explicit expression of the global minimizer of the sum of the linearization of the smooth part at the current iterate, the regularized function, and a quadratic proximal term over the unit sphere. Numerical experiments demonstrate the effectiveness of the proposed geometric algorithm.

\vspace{2mm}
\noindent {\bf Acknowledgments} The research of Z.-J. Bai was partially supported by the National Natural Science Foundation of China (No. 11671337).

\begin{appendices}
\section{}\label{app}
In this appendix, we give some preliminary results on subgradients of nonsmooth functions and Kurdyka-{\L}ojasiewicz (KL) property. We first recall the definition of lower semicontinuity in \cite{RW98,R70}.
\begin{definition}\label{app:lsc}
Let $h$ be a function from $\Rn$ to $[-\infty, \infty]$. Then $h$ is  proper if $\dom~h:=\{\ba\in\Rn\; | \; h(\ba)<\infty\}\neq\emptyset$ and $h(\ba)>-\infty$ for all $\ba\in\Rn$. Moreover, $h$ is lower semicontinuous (lsc) at $\bar{\ba}\in \Rn$ if
\[
\liminf_{\ba\to\bar{\ba}}h(\ba)\ge h(\bar{\ba}).
\]
and lower semicontinuous on $\Rn$ if it is lsc for every $\bar{\ba}\in \Rn$.
\end{definition}

Next, we recall some subdifferentials (subgradients) for nonsmooth functions in \cite{M06,RW98}.
\begin{definition}
Let $h:\Rn\to [-\infty, +\infty]$ be a proper lsc function. Then, the set
\[
\widehat{\partial} h(\bar{\ba}):=\left\{\br\in\Rn \; |\; \liminf_{\ba\to\bar{\ba}}\frac{h(\ba)-h(\bar{\ba})-\langle \br,\ba-\bar{\ba}\rangle }{\|\ba-\bar{\ba}\|}\geq 0\right\}.
\]
is the presubdifferential or Fr\'{e}chet subdifferential of $h$ at $\bar{\ba}\in \dom~h$ and we set $\widehat{\partial} h(\bar{\ba}):=\emptyset$ if $\bar{\ba}\notin \dom~h$. Moreover, the set
\BE\label{def:lsub}
\partial h(\bar{\ba}):=\left\{\br\in\Rn~|\mbox{~$\exists \ba^k\to\bar{\ba}$, $h(\ba^k)\to h(\bar{\ba})$ and $\br^k\in \widehat{\partial} h(\ba^k)\to\br$ as $k\to \infty$}\right\}
\EE
is the limiting subdifferential of $h$ at $\bar{\ba}\in\Rn$.
\end{definition}

As noted in \cite[Theorem 8.6]{RW98}, we know that, for each  $\bar{\ba}\in \dom~g$, $\widehat{\partial} h(\bar{\ba})\subset\partial h(\bar{\ba})$, where $\widehat{\partial} h(\bar{\ba})$ is convex and closed while  $\partial h(\bar{\ba})$ is closed. If $\bar{\ba}\in \Rn$ is a minimizer of $h$, then ${\bf 0}\in\partial h(\bar{\ba})$ and $\bar{\ba}$ is a critical point of $h$.

On the partial subdifferential of a nonsmooth function $F$ defined in \eqref{def:fg}, we have the following result from \cite{AB10,BS14,RW98}.
\begin{lemma}\label{app:ps}
Let $F$ be defined in  \eqref{def:fg}. Then for all $(\la,\by)$ with $\la>0$ and $\by\in\Rn$, we have
\[
\partial_\by F(\la,\by)=\{\nabla f(\by)+\nabla_\by g(\la,\by)+\partial\chi_{\cs^{n-1}}(\by)\}.
\]
\end{lemma}

We now recall the Kurdyka-{\L}ojasiewicz (KL) property for a nonsmooth function \cite{AB10,BS14}.
\begin{definition}\label{app:kli} (Kurdyka-{\L}ojasiewicz~property)
Let $h:\Rn\to [-\infty, +\infty]$ be a proper lsc function. Then $h$ is said to have the  Kurdyka-{\L}ojasiewicz property at $\bar{\ba}\in\dom~\partial h:=\{\ba\in\Rn\; |\; $ $\partial h(\ba)\neq\emptyset\}$ if there exist $\mu\in (0,+\infty]$, a neighborhood  $\cb$ of $\bar{\ba}$ and a function $\xi:[0,\mu)\to \R_+$ such that $\xi$ is concave and continuously differentiable on $(0,\mu)$ and continuous at $0$ with $\xi(0)=0$ and $\xi'(x)>0$ for all
$x\in(0,\mu)$, and
the  Kurdyka-{\L}ojasiewicz inequality
\[
\xi'\big(h(\ba)-h(\bar{\ba})\big) \dist({\bf 0}, \partial h(\bar{\ba}))\geq 1
\]
holds for all $\ba\in \cb\cap\{\ba\in\Rn\; | \; h(\bar{\ba})<h(\ba)<h(\bar{\ba})+\mu\}$, where $\dist({\bf 0}, \partial h(\bar{\ba})):=\inf\{\|\br\|\; | \; \br\in\partial h(\bar{\ba})\}$.
\end{definition}

Finally, we recall the general result from \cite[Lemma 6]{BS14} on the KL  property for a nonsmooth function.
\begin{lemma}\label{app:gkl}
Let $h:\Rn\to [-\infty, +\infty]$ be a proper lsc function. Suppose $h$ is constant on a compact set $\cc\subset \Rn$. If $h$ has the KL property at each point of $\cc$. Then, there exist two constants $\mu_1>0$ and $\mu_2>0$ and a function  $\xi$ as defined in Definition \ref{app:kli} such that
\[
\xi'\big(h(\ba)-h(\bar{\ba})\big) \dist({\bf 0}, \partial h(\bar{\ba}))\geq 1,
\]
for all $\bar{\ba}$ in $\cc$ and all $\ba\in\{\ba\in\Rn \; |\; \dist(\ba,\cc)<\mu_1\} \cap\{\ba\in\Rn \; |\; h(\bar{\ba})<h(\ba)<h(\bar{\ba})+\mu_2\}$.

\end{lemma}
\end{appendices}
\end{document}